\documentclass[a4paper]{amsart}
\usepackage{cite}

\usepackage{color}

\usepackage{graphics}
\usepackage{amsthm}
\usepackage{amsmath,amsfonts,amssymb,latexsym}
\usepackage{graphics,color}
\usepackage{color}		
\usepackage{subfigure}
\usepackage{epsfig}

\usepackage{mathrsfs}
\usepackage{pst-pdf}
\usepackage{latexsym}
\usepackage{psfrag}
\usepackage{verbatim}
\usepackage{graphicx}
\usepackage{xcolor}
\usepackage{pifont,marvosym}
\usepackage{psfrag}
\usepackage{mathrsfs}
\usepackage[bookmarksnumbered, bookmarksopen, colorlinks, citecolor=blue, linkcolor=blue]{hyperref}
\usepackage{color}

\theoremstyle{plain}
\newtheorem{theorem}{Theorem}[section]
\newtheorem{lemma}[theorem]{Lemma}

\newtheorem{corollary}[theorem]{Corollary}
\theoremstyle{remark}
\newtheorem{definition}[theorem]{\bf Definition}
\newtheorem{remark}[theorem]{Remark}

\numberwithin{equation}{section}

\DeclareMathOperator{\Lip}{Lip}
\DeclareMathOperator*{\conv}{conv}
\DeclareMathOperator*{\conc}{conc}
\newcommand*{\R}{\ensuremath{\mathbb{R}}}

\newcommand*{\Z}{\ensuremath{\mathbb{Z}}}
\newcommand*{\T}{\ensuremath{\mathscr{T}}}
\renewcommand*{\t}{\ensuremath{\mathcal{T}}}

\newcommand*{\Q}{\ensuremath{\mathcal{Q}}}
\renewcommand{\S}{\ensuremath{\mathcal{S}}}

\newcommand*{\I}{\ensuremath{\mathcal{I}}}
\newcommand*{\D}{\ensuremath{\mathcal{D}}}

\newcommand*{\ul}{\ensuremath{u^\mathrm{L}}}
\newcommand*{\ur}{\ensuremath{u^\mathrm{R}}}
\newcommand*{\unu}{\ensuremath{u_\nu}}
\newcommand*{\pnu}{\ensuremath{P_\nu}}
\newcommand*{\qnu}{\ensuremath{Q_\nu}}

\newcommand*{\gm}{\ensuremath{y^{k}_m}}
\newcommand*{\gnum}{\ensuremath{y^{k}_{m,\nu}}}
\newcommand*{\pq}{\ensuremath{\overline{P_\nu Q_\nu}}}
\newcommand*{\gnuone}{\ensuremath{y^{k_1}_{1,\nu}}}
\newcommand*{\muicnu}{\ensuremath{\mu^{IC}_\nu}}
\newcommand*{\muic}{\ensuremath{\mu^{IC}}}
\newcommand*{\li}{\ensuremath{\lambda_i}}
\newcommand*{\ulnu}{\ensuremath{u^L_\nu}}
\newcommand*{\urnu}{\ensuremath{u^R_\nu}}
\renewcommand*{\r}{\ensuremath{\tilde{r}}}

\newcommand*{\f}{\ensuremath{\tilde{f}}}
\begin{document}

\title[Global structure of solutions to PWGN HCLs]{Global structure of admissible BV solutions to piecewise genuinely nonlinear, strictly hyperbolic conservation laws in one space dimension}


\author{Stefano Bianchini}
\address{SISSA, via Bonomea 265, 34136 Trieste, ITALY}
\email{bianchin@sissa.it}
\urladdr{http://people.sissa.it/{\raise.17ex\hbox{$\scriptstyle\sim$}}bianchin}

\author{Lei Yu}
\address{SISSA, via Bonomea 265, 34136 Trieste, ITALY}
\email{yulei@sissa.it}


\begin{abstract}
The paper gives an accurate description of the qualitative structure of an admissible BV solution to a strictly hyperbolic, piecewise genuinely nonlinear system of conservation laws. We prove that there are a countable set $\Theta$ which contains all interaction points and a family of countably many Lipschitz curves $\T$ such that outside $\T\cup \Theta$ $u$ is continuous, and along the curves in $\T$, u has left and right limit except for points in $\Theta$. This extends the corresponding structural result in \cite{BL,Liu1} for admissible solutions.

The proof is based on approximate wave-front tracking solutions and a proper selection of discontinuity curves in the approximate solutions, which converge to curves covering the discontinuities in the exact solution $u$.
\end{abstract}

\maketitle

\section{Introduction}
This paper is concerned with the qualitative structure of admissible solutions to the strictly hyperbolic $N\times N$ system of conservation laws in one space dimension

\begin{equation}
\label{basic equation}
\begin{cases}
u_t+f(u)_x=0 & u:\R^+\times\R\to\Omega\subset\R^N,\ f\in C^2(\Omega,\R), \crcr
u_{|t=0}=u_0 & u_0\in \mathrm{BV}(\R;\Omega).
\end{cases}
\end{equation}

We assume the strict hyperbolicity in $\Omega$: the eigenvalues $\{\lambda_i(u)\}_{i=1}^N$ of the Jacobin matrix $A(u)=Df(u)$ satisfy
\begin{equation*}
\lambda_1(u)<\dots<\lambda_N(u), \qquad u \in\Omega.
\end{equation*}
Furthermore, as we only consider the solutions with small total variation, it is not restrictive to assume that $\Omega$ is bounded and there exist constants $\{\check{\lambda}_j\}^N_{j=0}$, such that
\begin{equation}\label{lambda}
\check{\lambda}_{k-1}<\lambda_k(u)<\check{\lambda}_{k}, \qquad \forall u\in\Omega,\ k=1,\dots, N.
\end{equation}

Let  $\{r_i(u)\}_{i=1}^N$ and $\{l_j(u)\}_{j=1}^N$ be a basis of right and left eigenvectors, depending smoothly on $u$, such that
\begin{equation*}\label{assumponri}
l_j(u) \cdot r_i(u) = \delta_{ij} \text{ and } |r_i(u)| \equiv 1, \quad i,j=1,\dots, N.
\end{equation*}

Let $R_i[u_0](\omega)$ be the value at time $\omega$ of the solution to the Cauchy problem
\[
 \frac{du}{d\omega}=r_i(u(\omega)),\quad u(0)=u_0.
\]
We call the curve $R_i[u_0]$ the \emph{$i$-rarefaction curve} through $u_0$.

We say that the system \eqref{basic equation} is \emph{piecewise genuinely nonlinear} if the set where $\nabla \lambda_i \cdot r_i = 0$ is covered by $\bar k_i$ transversal manifolds: more precisely,
\begin{equation*}
Z_i:= \big\{u \in \Omega:\ |\nabla\lambda_i\cdot r_i(u)|=0 \big\}=\bigcup^{\bar{k}_i}_{j=1} Z_i^j,
\end{equation*}
where $Z_i^j$ is a $N-1$-dimensional manifolds such that
\begin{enumerate}
\item each $Z^j_i$ is transversal to the vector field $r_i(u)$, i.e.
\begin{equation}
 \big{(} \nabla (\nabla\lambda_i\cdot r_i)\cdot r_i\big{)}(u)\ne 0 \quad \text{for $u\in Z_i^j$};
\end{equation}

\item each rarefaction curve $R_i[u_0]$ crosses all the $Z^j_i$, and moreover defining the points $\omega^j[u_0]$ by
\[
R_i[u_0](\omega^j[u_0]) \in Z_i^j,
\]
then $j \mapsto \omega^j[u_0]$ is strictly increasing.
\end{enumerate}
This implies that along $R_i$,  $\lambda_i$ has a finite number of critical points.

Denote with $\Delta^j_i$ the set of points $u$ between $Z^j_i$ and $Z^j_{i+1}$:
\[
\Delta^j_i:= \big\{ u\in \Omega: \ \omega^j[u] < 0 < \omega^{j+1}[u] \big\}.
\]
Without any loss of generality, we assume that
\begin{subequations}\label{gnhull}
\begin{equation}
 \nabla \lambda_i\cdot r_i(u)< 0\ \text{   if $j$ is even},\ u \in \Delta^j_i,
\end{equation}
\begin{equation}
 \nabla \lambda_i\cdot r_i(u)> 0\ \text{   if $j$ is odd},\ u \in \Delta^j_i.
\end{equation}
\end{subequations}
From now on, we assume that {\it every characteristic field of \eqref{basic equation} is piecewise genuinely nonlinear}.

It is well known that, because of the nonlinear dependence of the characteristic speeds $\lambda_i(u)$ on the state variable $u$, the solution to \eqref{basic equation} develops discontinuities within finite time, even with smooth initial data. Therefore, in order to construct solutions globally defined in time, one considers weak solutions interpreting the equation (1.1) in a distributional sense. We recall that $u\in C(\R^+; L^1_{loc}(\R;\R^N ))$ is a weak solution to the Cauchy problem \eqref{basic equation} if the initial condition is satisfied and, for any smooth function $\phi\in C^1_c(]0,T[\times \R)$ there holds
\[
\int^T_0 \int _{\R} \phi_t(t, x) u(t, x) + \phi_x (t, x)f(u(t, x))dxdt = 0
\]
As a consequence of the weak formulation, it follows that a function with a single jump discontinuity
\begin{equation*}
u(t, x)=\begin{cases}
         \ul & \text{if}\  x < \sigma t,\\
\ur &\text{if}\  x > \sigma t,
        \end{cases}
\end{equation*}
is a solution to \eqref{basic equation}, if and only if the left and right states $\ul,\ur \in \R^N$, and the speed $\hat \sigma$ satisfy the Rankine-Hugoniot condition
\begin{equation}\label{d:R-H condition}
f(\ur) -f(\ul ) = \hat \sigma(\ur -\ul ).
\end{equation}

By the strict hyperbolicity, it is known that for any $u^-\in \Omega$ there exists $s_0>0$ and N smooth curves $S_i[u^-]:[-s_0,s_0]\rightarrow \Omega$, associated with functions $\hat \sigma_i:[-s_0,s_0]\rightarrow \R$ such that
\begin{equation}\label{d:rh}
\hat \sigma_i(s)[S_i[u^-](s)-u^-]=f(\hat \sigma_i(S_i[u^-](s))-f(u^-)
\end{equation}
and satisfying
\begin{equation*}
 S_i[u^-](0)=u^-,\qquad \hat \sigma_i(0)=\li(u^-), \qquad \frac{d}{ds}S_i[u^-](0) = r_i(u^-).
\end{equation*}
The curve $S_i[u^-]$ the $i$-th \emph{Hugoniot curve} issuing from $u^-$ and we also say that $[u^-,u^+]$ is an \emph{$i$-discontinuity} with speed $\hat \sigma_i(u^-,u^+):=\hat \sigma(s)$ if $ u^+= S_i[u^-](s)$.

Since weak solutions to \eqref{basic equation} may not be unique, an entropy criterion for admissibility is usually added to rule out
nonphysical discontinuities. In \cite{Liu1}, T.P. Liu proposed the following admissibility criterion valid for weak solutions to general systems of conservation laws

We say the $i$-discontinuity $[u^-,u^+]$, $u^+ = S_i[u^-](s)]$, is \emph{Liu admissible} if it satisfies \emph{Liu admissible condition}: for $s_0>0$
\begin{equation*}
 \hat \sigma_i(u^+,u^-)\leq \hat \sigma_i(u,u^-)
\end{equation*}
where $u=S_i[u^-](\tau)$ for each $\tau \in ]0,s_0[$, and for $s_0 < 0$
\[
\hat \sigma_i(u^+,u^-) \geq \hat \sigma_i(u,u^-),
\]
where $u=S_i[u^-](\tau)$ for each $\tau\in]s_0,0[$.

Let  $[u^-,u^+]$, $u^+= S_i[u^-](s)$, be a Liu admissible $i$-discontinuity. Following the notation of \cite{Liu1}, we call the jump $[u^-,u^+]$ \emph{simple} if $\forall \tau\in ]0,s[$ \text{ when } $s>0\ (\forall \tau\in ]s,0[$ \text{ when } $s<0)$,
\begin{equation*}
 \hat \sigma_i(S_i[u^-](\tau),u^-)<\hat \sigma_i(u,u^-)\quad (\hat \sigma_i(S_i[u^-](\tau),u^-)>\hat \sigma_i(u,u^-)).
\end{equation*}
If $[u^-,u^+]$ is not simple, then we call it a \emph{composition} of the waves $[u^-,u_1],[u_1,u_2],\cdots,[u_l,u^+]$, if
\begin{equation}
\label{E_composi_point}
u_k=S_i[u^-](s_k) \quad \text{and} \quad \hat \sigma_i(u^k,u^{k-1})= \hat \sigma_i(u^+,u^-),
\end{equation}
where
\[
0=s_0<s_1<s_2<\cdots<s_l<s \quad (\text{or} \ s<s_l<\cdots<s_1<s_0=0),
\]
%
and there are no other points $\tau$ such that \eqref{E_composi_point} holds.

In \cite{Liu1}, under assumption of piecewise  genuinely nonlinearity, it is proved by using Glimm scheme that if the initial data has small total variation, there exists a weak BV solution of \eqref{basic equation} satisfying Liu admissible condition. Therefore, it enjoys the usual regularity properties of BV function: $u$ either is approximately continuous or has an approximate jump at each point $(x,t)\in \R^+\times\R\setminus \mathcal{N}$, where $\mathcal{N}$ is a subset whose one-dimensional Hausdorff measure is zero.

In \cite{Liu1}, the author shows much stronger regularity that $u$ holds. The set $\mathcal{N}$ contains at most countably many points. Moreover, $u$ is continuous (not just approximate continuous) outside $\mathcal{N}$ and countably many Lipschitz continuous curves.

In \cite{BL}, the authors adopt wave-front tracking approximation to prove the similar result for \eqref{basic equation} with the assumption that each characteristic field is genuinely nonlinear. Moreover, the authors where able to prove that outside the countable set $\mathcal N$ there exist right and left limits $u^-$, $u^+$ along the jump curves in the uniform norm, and these limits are stable w.r.t. wavefront approximate solutions: more precisely, for each jump point (not interaction point) of the solution, there exists a jump curve for the approximate solution converging to it and such that its left and right limit converge to $u^-$, and $u^+$ uniformly. In \cite{Bre} (Theorem 10.4), the author generalize his result in \cite{BL} to the case when some characteristic field may be linearly degenerate.

To prove this new regularity estimates one has to overtake additional difficulties, and this is the reason why they have so far been restricted to genuinely nonlinear of linearly degenerate systems: in fact the proof in \cite{Bre} is based on the wave structure of the solution to genuinely nonlinear or linearly degenerate systems, where only one shock curve passes through the discontinuous point (which is not an interaction point) of the admissible solution.

In this paper, we extend the techniques of \cite{Bre} to prove an analogous result about global structure of admissible solution for piecewise genuinely nonlinear system of \eqref{basic equation} by means of wave-front tracking approximation. This not only completes the corresponding result in \cite{Liu1}, but also makes it possible to prove SBV regularity for the solution of piecewise genuinely nonlinear strictly hyperbolic system. In fact, one of the key argument for SBV regularity in the proofs contained in \cite{BCa} and \cite{BY}, is that outside the interaction points the left and right values of jumps are approximated uniformly by wavefront approximate solutions.

\begin{figure}[htbp]
  \hfill
  \begin{minipage}[t]{.45\textwidth}
    \begin{center}
             \begin{picture}(0,0)%
\includegraphics{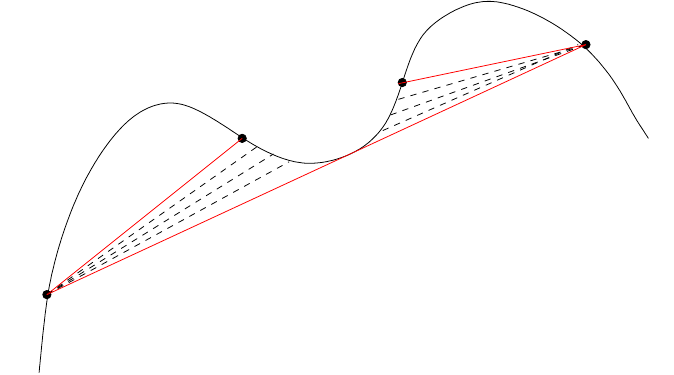}%
\end{picture}%
\setlength{\unitlength}{1973sp}%
\begingroup\makeatletter\ifx\SetFigFont\undefined%
\gdef\SetFigFont#1#2#3#4#5{%
  \reset@font\fontsize{#1}{#2pt}%
  \fontfamily{#3}\fontseries{#4}\fontshape{#5}%
  \selectfont}%
\fi\endgroup%
\begin{picture}(6465,3589)(2401,-4948)
\put(4126,-2011){\makebox(0,0)[lb]{\smash{{\SetFigFont{8}{9.6}{\familydefault}{\mddefault}{\updefault}{\color[rgb]{0,0,0}$f(u)$}%
}}}}
\put(2401,-4186){\makebox(0,0)[lb]{\smash{{\SetFigFont{8}{9.6}{\familydefault}{\mddefault}{\updefault}{\color[rgb]{0,0,0}$u_1$}%
}}}}
\put(4876,-2461){\makebox(0,0)[lb]{\smash{{\SetFigFont{8}{9.6}{\familydefault}{\mddefault}{\updefault}{\color[rgb]{0,0,0}$u_2$}%
}}}}
\put(5851,-2011){\makebox(0,0)[lb]{\smash{{\SetFigFont{8}{9.6}{\familydefault}{\mddefault}{\updefault}{\color[rgb]{0,0,0}$u_3$}%
}}}}
\put(8251,-1561){\makebox(0,0)[lb]{\smash{{\SetFigFont{8}{9.6}{\familydefault}{\mddefault}{\updefault}{\color[rgb]{0,0,0}$u_4$}%
}}}}
\end{picture}%
      \caption{}
      \label{f:tsirw}
    \end{center}
  \end{minipage}
  \hfill
  \begin{minipage}[t]{.45\textwidth}
    \begin{center}
      \begin{picture}(0,0)%
\includegraphics{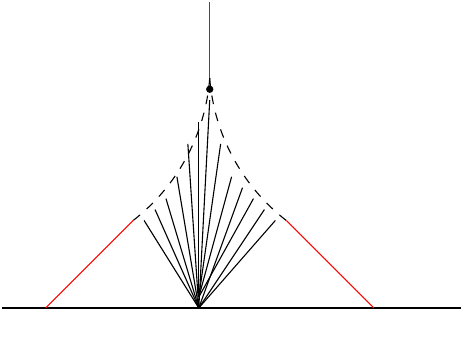}%
\end{picture}%
\setlength{\unitlength}{2763sp}%
\begingroup\makeatletter\ifx\SetFigFont\undefined%
\gdef\SetFigFont#1#2#3#4#5{%
  \reset@font\fontsize{#1}{#2pt}%
  \fontfamily{#3}\fontseries{#4}\fontshape{#5}%
  \selectfont}%
\fi\endgroup%
\begin{picture}(3174,2382)(4864,-4681)
\put(5176,-4636){\makebox(0,0)[lb]{\smash{{\SetFigFont{8}{9.6}{\familydefault}{\mddefault}{\updefault}{\color[rgb]{0,0,0}$x_1$}%
}}}}
\put(6151,-4636){\makebox(0,0)[lb]{\smash{{\SetFigFont{8}{9.6}{\familydefault}{\mddefault}{\updefault}{\color[rgb]{0,0,0}$x_2$}%
}}}}
\put(7276,-4636){\makebox(0,0)[lb]{\smash{{\SetFigFont{8}{9.6}{\familydefault}{\mddefault}{\updefault}{\color[rgb]{0,0,0}$x_3$}%
}}}}
\put(7276,-3811){\makebox(0,0)[lb]{\smash{{\SetFigFont{8}{9.6}{\familydefault}{\mddefault}{\updefault}{\color[rgb]{0,0,0}$u_4$}%
}}}}
\put(5176,-3886){\makebox(0,0)[lb]{\smash{{\SetFigFont{8}{9.6}{\familydefault}{\mddefault}{\updefault}{\color[rgb]{0,0,0}$u_1$}%
}}}}
\put(6451,-2986){\makebox(0,0)[lb]{\smash{{\SetFigFont{8}{9.6}{\familydefault}{\mddefault}{\updefault}{\color[rgb]{0,0,0}$P$}%
}}}}
\put(6676,-4261){\makebox(0,0)[lb]{\smash{{\SetFigFont{8}{9.6}{\familydefault}{\mddefault}{\updefault}{\color[rgb]{0,0,0}$u_3$}%
}}}}
\put(5701,-4261){\makebox(0,0)[lb]{\smash{{\SetFigFont{8}{9.6}{\familydefault}{\mddefault}{\updefault}{\color[rgb]{0,0,0}$u_2$}%
}}}}
\end{picture}%
      \caption{}
     \label{f:tsc}
    \end{center}
  \end{minipage}
  \hfill
\end{figure}

As we said, the assumption of genuinely nonlinearity done in \cite{BL} implies that there is only one shock curve passing through the discontinuous point (which is not an interaction point) of the admissible solution. For piecewise genuinely nonlinear case, however, due the presence of the composite discontinuity, there may be several discontinuity curves passing through the same discontinuity point which is not an interaction point. For example, consider a scalar equation where $f$ has two inflection points (it is thus clearly piecewise genuinely nonlinear), Figure \ref{f:tsirw}, and let $u_0$ be the initial data be
\begin{equation*}
 u_0=\begin{cases}
  u_1 & \mbox{if $x<x_1$},\\
  u_2 & \mbox{if $x_1<x<x_2$},\\
  u_3 & \mbox{if $x_2<x<x_3$,}\\
  u_4 & \mbox{if $x>x_3$.}
 \end{cases}
\end{equation*}

%


The same Figure \ref{f:tsc} shows  two shocks connecting value $u_1,u_2$ and $u_3,u_4$ respectively interact with a center rarefaction wave and eventually have the same speed at $P$ (which is thus \emph{not} and interaction point) and combine together becoming a single shock. Clearly such wave pattern can not happen if $f$ is convex or concave.

In this paper we prove the following theorem:
\begin{theorem}\label{t:main theorem}
 Let $u$ be a Liu admissible solution of the Cauchy problem \eqref{basic equation}. Then there exist a countable set $\Theta$ of interaction points and a countable family $\T$ of Lipschitz continuous curves such that $u$ is continuous outside $\Theta$ and ${Graph}(\T)$.

Moreover, suppose $u(t_0,x)$ is discontinuous at $x=x_0$, and $(t_0,x_0)\notin\Theta$. Write $u^L=u(t_0,x_0-),\ u^R=u(t_0,x_0+)$ and suppose that $\ur=S_i[\ul](s)$ with $s>0\ (s<0)$.
\begin{itemize}
 \item If $[u^L,u^R]$ is simple, there exists a Lipschitz curve $y\in \T$, s.t $y(t_0)=x_0$
\begin{equation*}
 \ul=\lim_{\genfrac{}{}{0pt}{}{x<y(t)} {(x,t)\rightarrow(t_0,x_0)}}u(x,t),\ \qquad \ur=\lim_{\genfrac{}{}{0pt}{}{x>y(t)} {(x,t)\rightarrow(t_0,x_0)}}u(x,t)
\end{equation*}
and the curve $y$ propagates with shock speed $\hat \sigma(\ul,\ur)$ at $(t_0,x_0)$, that is
\begin{equation*}
 \dot{y}(t_0)(\ur-\ul)=f(\ur)-f(\ul).
\end{equation*}
and
\begin{equation*}
 \dot{y}(t_0)\leq \hat \sigma_i(S_i[\ul](\tau),\ul),\ \forall \tau\in [0,s] \ (\dot{y}(t_0)\geq \hat \sigma_i(S_i[\ul](\tau),\ul),\ \forall \tau\in [s,0]).
\end{equation*}

\item If $[\ul,\ur]$ is a composition of $[\ul,u_1],\ [u_1,u_2],\cdots,[u_l,\ur]$, then there exists $p$ Lipschitz continuous curves $y_1,\cdots,y_{p}\in \T$, $p\leq l+1$ satisfying
\begin{itemize}
 \item[-] $y_1(t_0)=\cdots=y_{p}(t_0)=x_0$,
 \item[-] $y'_1(t_0)=\cdots=y'_{p}(t_0)$,
 \item[-] $y_1(t)\leq \cdots\leq y_{p}(t)$,for all $t$ in a neighborhood of $t_0$,
\end{itemize}
s.t.
\[
  \ul=\lim_{\genfrac{}{}{0pt}{}{x<y_1(t)} {(x,t)\rightarrow(t_0,x_0)}}u(x,t),\ \qquad \ur=\lim_{\genfrac{}{}{0pt}{}{x>y_{p}(t)} {(x,t)\rightarrow(t_0,x_0)}}u(x,t),
\]
and if in a small neighborhood of $(t_0,x_0)$, $y_j$ and $y_{j+1}$ are not identical, one has
\begin{equation}\label{e:lim inside}
u_j=\lim_{\genfrac{}{}{0pt}{}{y_j(t)<x<y_{j+1}(t)} {(x,t)\rightarrow(t_0,x_0)}}u(x,t).
\end{equation}
Also, these curves propagate with speed $\hat \sigma(\ul,\ur)$ at $(t_0,x_0)$, that is
\begin{equation*}
 \dot{y}_n(t_0)(\ur-\ul)=f(\ur)-f(\ul),\qquad n\in\{1,\cdots,p\}.
\end{equation*}
and the stability conditions hold:
\begin{equation*}
 \dot{y}_n(t_0)\leq \hat \sigma_i(S_i[\ul](\tau),\ul),\ \forall \tau\in [0,s], \ (\dot{y}_n(t_0)\geq \hat \sigma_i(S_i[\ul](\tau),\ul),\ \forall \tau\in [s,0]).
\end{equation*}
\end{itemize}
\end{theorem}

As in \cite{Bre}, the above result is based on this strong convergence result for approximate wave-front solutions.

\begin{theorem}\label{t:approx_shock_conv}
 Consider a sequence of wave-front tracking approximate solutions $u_\nu$ (see Section \ref{s:ft} for definitions) converging to $u$ in $L^1_{loc}$. Suppose $P=(\tau,\xi)$ is a discontinuous point of $u$ and write $u^L=u(\tau,\xi-),\ u^R=u(\tau,\xi+)$. Assume there are only $l$ Lipschitz continuous curves $\T\ni y_{n}:[t^-_{n},t^+_{n}]\mapsto\R,\ n=1,\cdots,l$ passing through the point $P$ and
\[
 y_1(t)\leq\cdots\leq y_l(t)\qquad \text{in a small neighborhood of $\tau$.}
\]
Then up to a subsequence, there exists $y_{n,\nu}:[t^-_{n,\nu},t^+_{n,\nu}]\mapsto\R,\ n=1,\cdots,l$ which are discontinuity curves of $u_\nu$ with uniformly large strengths, where $t^-_{n,\nu}\to t^-_n,t^+_{n,\nu}\to t^+_n$ and
\[
 y_{n,\nu}(t)\rightarrow y_n(t)\qquad \text{for every $t\in[t^-_n,t^+_n]$.}
\]
Moreover, one has
\begin{subequations}
\begin{equation*}
\lim_{r \rightarrow 0+} \limsup_{\nu \rightarrow \infty} \left( \sup_{\genfrac{}{}{0pt}{}{x<y_{1,\nu}(t)}{(x,t)\in B(P,r)}} \big| u_\nu(x,t) - u^L \big| \right) = 0,
\end{equation*}
\begin{equation*}
\lim_{r \rightarrow 0+} \limsup_{\nu \rightarrow \infty} \left( \sup_{\genfrac{}{}{0pt}{}{x>y_{l,\nu}(t)} {(x,t)\in B(P,r)}} \big| u_\nu(x,t) - u^R \big| \right) = 0.
\end{equation*}
\end{subequations}

\end{theorem}

A brief outline of this paper follows.

 In Section \ref{s:rp}, we recall the definition of construction of Riemann solvers introduced in \cite{BB}.

In Section \ref{s:ft}, we briefly describe the wave-front tracking approximate scheme which originally designed for general strictly hyperbolic system (see \cite{AM1}). In particular, we introduce the definition of interaction and cancelation measures. 

Section \ref{s:sdc} contains the main idea of the paper: the definition of subdiscontinuity curves and $(\epsilon,k)$-approximate subdiscontinuity curves in the approximate wavefront solution. In this section we show that their number is uniformly bounded with respect to approximation parameter.

In Section \ref{s:pf}, we finally give the proof for Theorem \ref{t:main theorem} and Theorem \ref{t:approx_shock_conv}. 

In Section \ref{s:example}, we construct a strictly hyperbolic $2\times 2$ system of conservation laws, which is not piecewise genuinely nonlinear and its admissible solutions to some initial datum do not have the structural properties described in Theorem \ref{t:main theorem}.

\section{Solution of Riemann problem}\label{s:rp}
As in \cite{B2}, for a fixed point $u^0\in \Omega$ and $i\in\{1,\cdots,N\}$, one can construct smooth vector-value maps $\r_i=\r_j(u,v_i,\sigma_i)$ for $(u,v_i,\sigma_i)\in \R^N\times\R\times\R$ with $\r_i(u,0,\sigma)=r_i(u)$ for all $u,\ \sigma_i$. Setting $l^0_i:=l_i(u^0)$, we can normalize $\r_i$ such that
\begin{equation}\label{a:para_Tk}
\langle l^0_j,\r_i(u,v_i,\sigma_i)\rangle=\begin{cases}
                                       1 & i=j,\\
                                       0 & i\ne j.
                                      \end{cases}
\end{equation}

Writing the speed function $\tilde{\lambda}_i:=l^0_i\cdot Df(u) \r_i(u,v_i,\sigma_i)$, we consider the set, for some $\delta_0,C_0>0$ fixed and $s>0$
\begin{equation*}
 \begin{split}
\Gamma_i(s,u^-):=&\Big{\{}\gamma \in \Lip([0,s],\R^{N+2}),\ \gamma(\xi)=(u(\xi),v_i(\xi),\sigma_i(\xi))\\
&\ u(0)=u^-,|u(\tau)-u^-|=\tau, v_i(0)=0,\\
&\ |v_i|\leq \delta_1, |\sigma_i(\tau)-\lambda_i(u^0)|\leq 2C_0\delta_1\Big{\}}.
\end{split}
\end{equation*}

Given a curve $\gamma\in \Gamma_i$, we define the scalar flux function
\begin{equation}\label{d:scalar_flux}
 \f_i(\tau,\;\gamma)=\int^\tau_0 \tilde{\lambda}_i(u(\xi),v_i(\xi),\sigma_i(\xi))d\xi.
\end{equation}
  Moreover, we define the lower convex envelope of $\f_i$ on $[a,b]\subset [0,s]$ as
\[
\begin{split}
\conv_{[a,b]}\f_i(\tau;\gamma):=\inf \Big{\{} \theta f_i(\tau',\gamma)&+(1+\theta)f_i(\tau'',\gamma);\\ &\theta\in[0,1],\tau',\tau''\in[a,b],\tau=\theta\tau'+(1+\theta)\tau''\Big{\}}.
\end{split}
\]

 Then define a nonlinear operator $\t_{i,s}:\Gamma_i(s,u^-)\to\Gamma_i(s,u^-)$ by setting $\gamma=\check \gamma:=(\check{u},\check{v}_i,\check{\sigma}_i)$,where
\begin{equation}\label{d:element operator}
 \begin{cases}
\check{u}(\tau)=u^-+\int^\tau_0\r_i(u(\xi),v_i(\xi),\sigma(\xi))d\xi,\\
\check{v}_i(\tau)=\tilde{f}_i(\tau;\gamma)-\conv_{[0,s]}\tilde{f}_i(\tau,\gamma),\\
\check{\sigma}_i=\frac{d}{d\tau}\conv_{[0,s]}\f_i(\tau;\gamma).
\end{cases}
\end{equation}

One can show that $\t$ is a contraction in $\Gamma_i(s,u^-)$ with respect to the distance
\begin{equation*}
 D(\gamma,\gamma'):=\delta_1||u-u'||_{L^\infty}+||v_i-v'_i||_{L^1}+||v_i\sigma_i-v'_i\sigma'_i||_{L^1},
\end{equation*}
where
\begin{equation*}
 \gamma=(u,v_i,\sigma_i),\ \gamma'=(u',v'_i,\sigma'_i)\in \Gamma_i(s,u^-).
\end{equation*}

Hence, for any $s$ and $u^-$ in a small neighborhood of $u^0$, $\t_i$ has a unique fixed point, which is a Lipschitz continuous curve
\begin{equation*}
 \bar{\gamma}(\tau)= \big{(}\bar{u}(\tau;u^-,s),\bar{v}_i(\tau;u^-,s),\bar{\sigma}_i(\tau;u^-,s)\big{)},\quad \tau\in[0,s].
\end{equation*}
Then the elementary curve for $i$-th family is defined as
\begin{equation}\label{d:Tk}
 T_i[u^-](s):=\bar{u}(s;u^-,s).
\end{equation}

After adopting the following notations
\begin{equation*}
\sigma_i[u^-](s,\tau):=\bar{\sigma}_i(\tau;u^-,s),
\end{equation*}
\begin{equation}\label{d:scalar function}
\f_i[u^-](s,\tau):=\f_i(\tau;\bar{\gamma}).
\end{equation}
and recalling that the Riemann problem is the Cauchy problem \eqref{basic equation} with piecewise constant  initial data of the form
\begin{equation}\label{e:r}
 u_0(x)=\left\{\begin{array}{ll}
u^\mathrm{L} & \mbox{$x<0$,}\\
u^\mathrm{R} & \mbox{$x>0$,}
\end{array}\right.
\end{equation}
where $\ul,\ \ur$ are two constants. One has the following theorem \cite{BB}.

\begin{theorem}
\label{t:ec}
For every $u\in\Omega$ and  $s>0$ sufficiently small,
\begin{enumerate}
\item N Lipschitz continuous curves $s \mapsto T_i[u](s)\in\Omega,\ i=1,\dots,N$, satisfying $\lim_{s\rightarrow 0}\frac{d}{ds}T_i[u](s)=r_i(u)$,
\item  N Lipschitz continuous functions $(s,\tau) \mapsto \sigma_i[u](s,\tau)$, with $0 \leq \tau \leq s$ and $i=1,\dots,N$, satisfying $\tau \mapsto \sigma_i[u](s,\tau)$ which are increasing and such that $\sigma_i[u](s,0) = \lambda_i(u)$,
\end{enumerate}
with the following properties.

\noindent When $u^\mathrm{L}\in\Omega,\ u^\mathrm{R}=T_i[\ul](s)$, for some $s$ sufficiently small, the unique Liu admissible solution of the Riemann problem \eqref{basic equation}-\eqref{e:r} is defined a.e. by
\begin{equation}\label{d:Riem_sol}
u(x,t) :=
\begin{cases} 
u^\mathrm{L} & x/t < \sigma_i[u^\mathrm{L}](s,0), \crcr
T_i[u^\mathrm{L}](\tau) & x/t=\sigma_i[u^\mathrm{L}](s,\tau), \tau \in [0,s], \crcr
u^\mathrm{R} & x/t>\sigma_i[u^\mathrm{L}](s,s).
\end{cases}
\end{equation}
\end{theorem}

\vspace{12pt}
For the case when $s<0$, a right state $\ur=T_i[\ul](s)$ can be constructed in the same way as before, except that one replaces $\conv_{[0,s]} \f_i$ in \eqref{d:element operator} with the upper concave envelope of $\f_i$ on $[s,0]$:
\[
\begin{split}
\conc_{[a,b]}\f_i(\tau;\gamma):=\sup \Big{\{} \theta f_i(\tau',\gamma)&+(1+\theta)f_i(\tau'',\gamma);\\ &\theta\in[0,1],\tau',\tau''\in[a,b],\tau=\theta\tau'+(1+\theta)\tau''\Big{\}},
\end{split}
\]
and looks at the fixed point of of the integral system \eqref{d:element operator} on the interval $[s,0]$.

Because of the assumption \eqref{a:para_Tk} and the definition \eqref{d:Tk}, the elementary curve $T_i[\ul]$ is parameterized by
its $i$-th component relative to the basis $r_1(u^0),\cdots,r_N(u^0)$ i.e.
\begin{equation}\label{e:para_T_k}
s = \langle l_i^0, T_i[\ul](s)-\ul\rangle.
\end{equation}

\begin{remark}
In \cite{BB}, it is proved that if $\ul,\ur\in \Delta^k_i$ with some $k$ even (odd) and $\ur=T_i[\ul](s),\ s>0\ (s<0)$, the solution $u$ of the Riemann problem with the initial date \eqref{e:r} is a \emph{center rarefaction wave}, that is for $t>0$,
\begin{equation*}
 u(x,t)=\begin{cases}
  \ul &\text{if}\ x/t<\lambda_i(\ul),\\
  \ur &\text{if}\ x/t>\lambda_i(\ur),\\
  R_i[\ul](\tau) &\text{if}\ x/t\in[\lambda_i(\ul),\lambda_i(\ur)],\ x/t=\lambda_i(R_i[\ul](\tau)),
 \end{cases}
\end{equation*}
where $\tau\in [0,s]\ (\tau\in [s,0])$ such that $s = \langle l_i^0, R_i[\ul](s)-\ul\rangle$. Notice that $u$ is smooth for $t>0$.
\end{remark}

\begin{remark}\label{r:shocks_rarefaction}
As shown in \cite{BB} (also see Remark 4 in \cite{AM1} and Section 4 of \cite{Liu1}), under the assumption of piecewise genuine nonlinearity, the solution of the Riemann problem provided by \eqref{d:Riem_sol} is a composed wave of the $i$-th family containing a finite number of rarefaction waves and Liu admissible discontinuities. Recalling Theorem \ref{t:ec}, one knows that the regions where the $v_i$-component of the solution to \eqref{d:element operator} vanishes correspond to rarefaction waves, while the regions where the $v_i$-component of the solution to \eqref{d:element operator} is different from zero correspond to admissible discontinuities.

\end{remark}
\vspace{12pt}


The Liu admissible solution \cite{BB} of a Riemann problem for \eqref{basic equation}-\eqref{e:r} is obtained by constructing a Lipschitz continuous map
\begin{equation}\label{RMap}
\mathbf{s}:=(s_1,\dots,s_N)\mapsto T[u^L](\mathbf{s}):=T_N\big[T_{N-1}\big[\cdots\left[T_1[u^\mathrm{L}](s_{1})\right]\cdots\big](s_{N-1})\big](s_N)=u^\mathrm{R},
\end{equation}
which is one to one from a neighborhood of the origin onto a neighborhood of $u^\mathrm{L}$. Then we can uniquely determine intermediate states $u^\mathrm{L}=\omega_0,\ \omega_1,\ \dots,\ \omega_N = u^\mathrm{R}$, and the \emph{wave strength} $s_1,\ s_2,\ \dots,\ s_N$ such that
\begin{equation*}
\omega_i = T_i[\omega_{i-1}](s_i), \quad i=1,\dots,N,
\end{equation*}
provided that $|u^\mathrm{L}-u^\mathrm{R}|$ is sufficiently small.

By Theorem \ref{t:ec}, each Riemann problem with initial date
\begin{equation}\label{e:erp}
u_0 =
\begin{cases}
\omega_{i-1} & x<0, \\
\omega_i & x>0,
\end{cases}
\end{equation}
admits a self-similar solution $u_i$, containing only $i$-waves. 
We call $u_i$ the $i$-th \emph{elementary composite wave} or simply $i$-$wave$. Therefore, under the strict hyperbolicity assumption, the solution of the Riemann problem with the initial data \eqref{e:r} is obtained by piecing together the self-similar solutions of the Riemann problems given by \eqref{basic equation}-\eqref{e:erp}.

Indeed, from the strict hyperbolicity assumption \eqref{lambda}, the speed of each elementary $i$-th wave in the solution $u_i$ is inside the interval $[\check{\lambda}_{i-1},\check{\lambda}_{i}]$ if $s \ll 1$, so that the solution of the general Riemann problem \eqref{basic equation}-\eqref{e:r} is then given by
\begin{equation}
\label{e:riemann solution}
u(x,t) =
\begin{cases}
         u^\mathrm{L} & x/t <\check{\lambda}_{0},\\
         u_i(x,t) & \check{\lambda}_{i-1}<x/t<\check{\lambda}_{i}, i=1,\dots,N,\\
         u^\mathrm{R} & x/t>\check{\lambda}_{N}.
         \end{cases}
\end{equation}


\section{Description of wave-front tracking approximation}\label{s:ft}
In \cite{AM1}, the authors provide an algorithm of wave-front tracking approximation for vanishing viscosity BV solutions to the strictly hyperbolic system which is much more general than the case discussed here. We modify a little about its algorithm in order to simplify our analysis. Due to Theorem \ref{t:ec}, one knows that the solution constructed by such approximation is Liu admissible.

Wave-front tracking approximation is an algorithm which produces  piecewise constant approximate solutions to the Cauchy problem \eqref{basic equation}. In order to construct approximate wave-front tracking solutions, given a fixed $\epsilon>0$, we first choose a piecewise constant function $u^\epsilon_0$ which is a good approximation to initial data $u_0$ such that
\begin{equation}\label{initial approx}
\mathrm{Tot.Var.}\{u^\epsilon_0\}\leq \mathrm{Tot.Var.}\{u_0\}, \quad ||u^\epsilon_0-u_0||_{L^1}<\epsilon,
\end{equation}
and $u^\epsilon_0$ only has finite jumps.  Let $x_1<\dots<x_m$ be the jump points of $u^\epsilon_0$. For each $\alpha=1,\dots,m$, we approximately solve the Riemann problem (just shifting the center from $(0,0)$ to $(0,x_\alpha)$) with the initial data of the jump $[u^\epsilon_0(x_\alpha-),[u^\epsilon_0(x_\alpha+)]$ by a function $w(x,t)=\phi(\frac{x-x_0}{t})$ where $\phi$ is a piecewise constant function. The straight lines where the discontinuities locate are called $wave$-$fronts$ (or just \emph{fronts} for short). The wave-fronts can prolong until they interact with other fronts, then at the interaction point, the corresponding Riemann problem is approximately solved and several new fronts are generated forward. Then one tracks the wave-fronts until they interact with other wave-fronts, etc... In order to avoid the algorithm to produce infinite many wave-fronts in finite time, different kinds of approximate Riemann solvers should be introduced.

\subsubsection{The approximate $i$-th elementary wave}
\label{Ss_k_gnl}

Suppose that $u_i$ is an $i$-th elementary composite wave which is obtain by solving Riemann problem with initial data \eqref{e:erp} where $\omega_i=T_i[\omega_{i-1}](s_i)$. For notational convenience, we write $\sigma_i(\tau):=\sigma_i[\omega_{i-1}](s_i,\tau)$. Let
\begin{equation*}
 p:=\left[\frac{\sigma_i(s_i)-\sigma_i(0)}{\epsilon}\right]+1
\end{equation*}
and
\begin{equation*}
 \vartheta_l:=\sigma_i(0)+\frac{l}{p}\left[\sigma_i(s_i)-\sigma_i(0)\right],\qquad l=0.\cdots,p-1.
\end{equation*}
We set
\begin{equation*}
 \omega_{i-1,l}=T_i[\omega_{i-1}](s_{i,l}),
\end{equation*}
where
\begin{equation*}
 s_{i,l}:=\begin{cases}
           \text{min}\big{\{}s\in [0,s_i], \sigma_i(s)=\vartheta_l \big{\}},\quad s_i\geq 0,\\
           \text{max}\big{\{}s\in [s_i,0], \sigma_i(s)=\vartheta_l \big{\}},\quad s_i\leq 0.
          \end{cases}
\end{equation*}

Then the $i$-th elementary composite wave $u_i$ is approximated by $\tilde{u}_i$ as the following,
\begin{equation}\label{e:aew}
   \tilde{u}_i(x,t)=\begin{cases}
      \omega_{i-1}\quad &x/t<\vartheta_{i,0},\\
      \omega_{i-1,l},\quad &\vartheta_{i-1,l-1}<x/t<\vartheta_{i,l},\quad (l=1,\cdots,p-1),\\
      \omega_{i}\quad &x/t>\vartheta_{i.p-1}.
     \end{cases}
\end{equation}
 Notice that $\tilde{u}_i$ consists of p fronts with small strength.

\subsubsection{Approximate Riemann solver}
Suppose at the point $(t_1,x_1)$, a wave-front $[u^L,u^M]$ of strength $s'$ belonging to $i'$-th family interacts from the left with a wave-front $[u^M,u^R]$ of strength $s''$ belonging to $i''$-th family for some $i',\ i''\in \{1,\cdots,N\}$ such that
\[
u^M=T_{i'}[u^\mathrm{L}](s'),\qquad u^\mathrm{R}=T_{i''}[u^M](s'').
\]
 Assume that $|u^\mathrm{L}-u^\mathrm{R}|$ sufficiently small. Then at the interaction point, the Riemann problem with the initial data of the jump  $[u^\mathrm{L},u^\mathrm{R}]$ may be solved by two kinds approximate Riemann solver according to different situation.

\begin{itemize}
 \item\emph{Accurate Riemann solver}: It replaces each elementary composite wave of the exact Riemann solution (refers to $u_i$ in  the solution \eqref{e:riemann solution}) with an approximate $i$-th elementary wave defined by \eqref{e:aew}.

\vspace{6pt}

 \item\emph{Simplified Riemann solver}: It only generates the approximate elementary waves belong to $i'$-th and $i''$-th families with the corresponding strength $s'$ and $s''$ as the incoming ones if $i'\ne i''$ or the approximate  $i'$-th elementary waves of strength $s'+s''$ if $i'= i''$. The simplified Riemann solver collects the remaining new waves into a single \emph{nonphysical front}, traveling with a constant speed $\hat{\lambda}$, strictly larger than all characteristic speeds. Therefore, usually the simplified Riemann solver generates less outgoing fronts after an interaction than the accurate Riemann solver.
\end{itemize}

Since the simplified Riemann solver produces nonphysical wave-fronts and they can not interact with each other, one needs an approximate Riemann solver defined for the  interaction between, for example, a physical front of the $i$-th family with strength $s$, connecting $u^M$, $u^\mathrm{R}$ and a nonphysical front (coming from the left)  connecting the left value $u^\mathrm{L}$ and $u^M$ traveling with speed $\hat{\lambda}$.
\begin{itemize}
\item\emph{Crude Riemann solver}: It generates an approximate $i$-th elementary wave connecting $u^\mathrm{L}$ and $\tilde{u}^M=T_i[u^\mathrm{L}](s)$ and a nonphysical wave-front joining $\tilde{u}^M$ and $u^\mathrm{R}$, traveling with speed $\hat{\lambda}$.  In the following, for simplicity, we just say that the non-physical fronts belong to the $(N+1)$-th characteristic field.
\end{itemize}

\begin{remark}
It is not restrictive to assume that at each time $t>0$, at most one interaction takes place, involving exactly two incoming fronts, because one can always slightly change the speeds of the incoming fronts if more than two fronts meet at the same point. It is sufficient to require that the error vanishes when the approximation solutions converge to the exact solution. Actually, suppose x=y(t) is a front in an approximate solution $u$ with parameter $\epsilon$ and $u^L=u(t,x-)$ and $u^R=u(t,x+)$ such that
\begin{equation}\label{R_H}
 u^R=T_i[u^L](s)
\end{equation}
for some index $i\in \{1,\cdots,N\}$ and wave strength $s$. Then the following holds
\begin{equation}\label{front speed error}
 \left|\dot{y}(t)-\sigma_i[u^L](s,\tau)\right|\leq 2\epsilon, \ \forall \tau\in [0,s].
\end{equation}
\end{remark}

\begin{remark}
 There are three kinds of physical wave-fronts. Suppose $[\ul,\ur]$ is wave front in an approximate solution, and $\ur=T_i[\ul](s)\ (s>0)$. Recalling the notaton \eqref{d:scalar function} and Remark \ref{r:shocks_rarefaction}, and writing $\tilde{f}_i(\tau):=\tilde{f}_i[\ul](s,\tau)$, one knows that $[\ul,\ur]$ may be one of the following three kinds of fronts:
\begin{itemize}
 \item \emph{Discontinuity front} if $\tilde{f}_i(\tau)>\conv_{]0,s[} \tilde{f}_i(\tau),\ \forall \tau\in]0,s[\setminus\mathcal{S}$, where $\mathcal{S}$ is set of finite cardinality.
 \item \emph{Rarefaction front} if $\tilde{f}_i(\tau)=\conv_{]0,s[} \tilde{f}_i(\tau),\ \forall \tau\in]0,s[$. In this case, $\ul,\ur \in \Delta^k_i$ for some $k$ even and $\ur=R_i[\ul](s)$.
 \item \emph{Mixed front} if there exist $]a,b[\subsetneq ]0,s[$ such that $\tilde{f}_i(\tau)>\conv_{]0,s[} \tilde{f}_i(\tau),\ \forall \tau\in]a,b[$ and $]c,d[\subsetneq ]0,s[$ such that $\tilde{f}_i(\tau)=\conv_{]0,s[} \tilde{f}_i(\tau),\ \forall \tau\in]c,d[$.
\end{itemize}

\end{remark}

\subsubsection{Interaction potential and BV estimates}
\label{interaction amount}

In order to check the total variations of approximate solutions are uniformly bounded with respect to time, one needs the estimate on the difference between the strength of the incoming waves and the strength of the outgoing waves produced by an interaction. Suppose two wave-fronts with strength $s'$ and $s''$ interact and $\f'_i,\ \f''_i$ are the corresponding scalar flux function defined by \eqref{d:scalar_flux}. We define the \emph{amount of interaction $\mathcal{I}(s',s'')$} between $s'$ and $s''$.

When $s'$ and $s''$ belong to different characteristic families, set
\begin{equation}\label{d:interation amount ijwaves}
\mathcal{I}(s',s'')=|s's''|.
\end{equation}

When $s',\ s''$ belong to the same family,
\begin{itemize}
 \item[(a)] If $s''>0$, we set
\begin{equation*}
\begin{split}
 \I(s',s''):=\int^{s'}_0&\big{|}\conv_{[0,s']}\f'_i(\xi)-\conv_{[0,s'+s'']}(\f'_i\cup\f''_i)(\xi)\big{|}d\xi\\
&+\int^{s'+s''}_{s'}\big{|}\conv_{[0,s'']}(\f'(s')\f''_i(\xi-s'))-\conv_{[0,s'+s'']}(\f'_i\cup\f''_i)(\xi)\big{|}d\xi,
\end{split}
\end{equation*}

\item[(b)] if $-s'\leq s''<0$, we set
\begin{equation*}
\begin{split}
 \I(s',s''):=\int^{s'+s''}_0&\big{|}\conv_{[0,s']}\f'_i(\xi)-\conv_{[0,s'+s'']}\f'_i(\xi)\big{|}d\xi\\
&+\int^{s'}_{s'+s''}\big{|}\conv_{[0,s']}\f'(\xi)-\conc_{[s'+s'',s']}\f'_i(\xi)\big{|}d\xi,
\end{split}
\end{equation*}
\item[(c)] if $s''<-s'$, we set
\begin{equation*}
\begin{split}
 \I(s',s''):=\int^{-s'}_{s''}&\big{|}\conc_{[s'',0]}\f''_i(\xi)-\conc_{[s'',-s']}\f''_i(\xi)\big{|}d\xi\\
&+\int^{0}_{-s'}\big{|}\conv_{[s'',0]}(\f'(\xi))-\conv_{[-s',0]}\f''_i(\xi)\big{|}d\xi,
\end{split}
\end{equation*}
\end{itemize}

Throughout the paper, we write $A\lesssim B\ (A\gtrsim B)$ if there exists a constant $C>0$ which only depends on the system \eqref{basic equation} such that $A\leq CB\ (A\geq CB)$.

Recall the Lipschitz continuous map $T$ defined in \eqref{RMap} and suppose $u^M=T_i[u^L](s_1),u^R=T_j[u^M](s_2)$ and $u^R=T[u^L](\mathbf{s})$. By Glimm's interaction estimates proved in \cite{B2} (also see Lemma 1 in \cite{AM1}), one has
\begin{equation}\label{gie}
 |\mathbf{s}-\mathbf{s}_1-\mathbf{s}_2|\lesssim  \I(s_1,s_2)
\end{equation}
where $\mathbf{s}_h=(s_1,s_2,\cdots,s_N)$, $h=1,2$ with $s_n=0$ for $n\ne h$.

At each time $t>0$ when no interaction occurs, and the approximate solution $u$ has jumps at $x_1,\dots,x_m$, we denote by
\[
\omega_1,\dots,\omega_m, \quad s_1,\dots,s_m, \quad i_1,\dots,i_m,
\]
their left states, signed strengths and characteristic families respectively: the sign of $s_\alpha$ is given by the respective orientation of $dT_i[u](s)/ds$ and $r_i$, if the jump at $x_\alpha$ belongs to the $i$-th family. The total variation of $u_\nu$ will be computed as
\[
V(t) := \sum_{\alpha} \big| s_\alpha \big|.
\]

Since $T_i[u^0]$ is a Lipschitz continuous function and the Lipschitz constant is uniformly bounded for any $u^0\in \Omega$, one has
\begin{equation}\label{e:TV_V}
{\rm Tot.Var}\{u(\cdot,t)\}\lesssim V(t).
\end{equation}
Then the estimate of increasing of ${\rm Tot.Var}\{u_\nu(\cdot,t)\}$ turns out to be the estimate of the total amount of interaction. Following \cite{B2}, we define the \emph{Glimm wave interaction potential} as follows:
\begin{equation*}\label{d:gp}
\begin{split}
\Q(t) &:= \sum_{\genfrac{}{}{0pt}{}{i_\alpha>i_\beta}{x_\alpha<x_\beta}} \big| s_\alpha s_\beta \big| + \frac{1}{4} \sum_{i_\alpha=i_\beta<N+1} \int^{|s_\alpha|}_0\int^{|s_\beta|}_0 \big| \sigma_{i_\beta}[\omega_\beta](s_\beta,\tau'')-\sigma_{i_\alpha}[\omega_\alpha](s_\alpha,\tau') \big| d\tau'd\tau''.
\end{split}
\end{equation*}

Denoting the time jumps of the total variation and the Glimm potential as
\[
\Delta V(\tau)=V(\tau+)-V(\tau-),\ \ \Delta \Q(\tau)=\Q(\tau+)-\Q(\tau-),
\]
the fundamental estimates are the following (Lemma 5 in \cite{AM1}): in fact, when two wave-fronts with strength $s',\ s''$ interact,
\begin{subequations}
\label{e:gpe_ve}
\begin{equation}
\label{e:gpe}
\Delta\Q(\tau)\lesssim\mathcal{I}(s',s''),
\end{equation}
\begin{equation}
\label{e:ve}
\Delta V(\tau)\lesssim \mathcal{I}(s',s'').
\end{equation}
\end{subequations}
Thus one defines the \emph{Glimm functional}
\begin{equation*}
\label{e:glimm_funct}
\Upsilon(t) := V(t) + C_0 \Q(t)
\end{equation*}
with $C_0$ suitable constant, so that $\Upsilon$ decreases at any interaction. Using this functional, one can prove that their total variations are uniformly bounded. Moreover, one can also show that the number of wave-fronts remains finite for all time (see section 6.1 of \cite{AM1}). This makes sense for the construction of approximate wave-front tracking solutions.

\subsubsection{Construction of the approximate solutions and their convergence to exact solution}

The construction starts at initial time $t=0$ with a given $\epsilon>0$, by taking  $u_{0,\epsilon}$ as a suitable piecewise constant approximation of initial data $u_0$, satisfying \eqref{initial approx}. At the jump points of $u_{0,\epsilon}$, we locally solve the Riemann problem by accurate Riemann solver. The approximate solution then can be prolonged until a first time $t_1$ when two wave-fronts interact. Again we solve the Riemann problem at the interaction point by an approximate Riemann solver. Whenever the amount of interaction (see Section \ref{interaction amount} for the definition) of the incoming waves is larger than some threshold parameter $\rho = \rho(\epsilon) > 0$, we shall adopt the accurate Riemann solver. Instead, in the case where the amount of interaction of the incoming waves is less than $\rho$, we shall adopt the simplified Riemann solvers. The threshold $\rho$ is suitably chosen so that the number of wave-fronts remains finite for all times. And we will apply the
crude Riemann solver if one of the incoming wave-front is non-physical front. One can show that the number of wave fronts is uniformly bounded  (see Section 6.2 in \cite{AM1}).

We call such approximate solutions \emph{$\epsilon$-approximate front tracking solutions}. At each time $t$ when there is no interaction, the restriction $u_\epsilon(t)$ is a step function whose jumps are located along straight lines in the $(x,t)$-plane. 

Let $\{\epsilon_\nu\}^\infty_{\nu=1}$ be a sequence of positive real numbers converging to zero. Consider a corresponding sequence of $\epsilon_\nu$-approximate front tracking solutions $u_\nu:=u_{\epsilon_\nu}$ of \eqref{basic equation}: it is standard to show that the functions $t\mapsto u_\nu(t,\cdot)$ are uniformly Lipschitz continuous in $L^1$ norm. since \eqref{e:TV_V} and \eqref{e:gpe_ve} hold independent of the parameter $\epsilon_\nu$, $u_\nu(t,\cdot)$ have uniformly bounded total variation. Therefore by Helly's theorem, $u_\nu$ converges up to a subsequence in $\mathbb{L}^1_{\mathrm{loc}}(R^+\times\R)$ to some function $u$, which is a weak solution of \eqref{basic equation}.

It can be shown that by the choice of the Riemann solver in Theorem \ref{t:ec}, the solution obtained by the front tracking approximation coincides with the unique vanishing viscosity solution \cite{BB}. Furthermore, there exists a closed domain $\D\subset L^1(\R,\Omega)$ and a unique distributional solution $u$, which is a Lipschitz semigroup $\D\times[0,+\infty[\rightarrow \D$ and which for piecewise constant initial data coincides, for a small time, with the solution of the Cauchy problem obtained piecing together the standard entropy solutions of the Riemann problems. Moreover, it lives in the space of BV functions.

For simplicity, the pointwise value of $u$ is its $L^1$ representative such that the restriction map $t\mapsto u(t)$ is continuous form the right in $L^1$ and $x \mapsto u(x,t)$ is right continuous from the right.

\subsubsection{Further estimates}

%
To each $u_\nu$, we define the \emph{measure $\mu^\mathrm{I}_\nu$ of interaction} and the \emph{measure $\mu^\mathrm{IC}_\nu$ of interaction and cancelation} concentrated on the set of interaction points as follows. If two physical fronts belonging to the families $i',i''\in\{1,\dots,N\}$ with strength $s',\ s''$ interact at point $P$, we denote

\begin{subequations}\label{ICmeas}
 \begin{equation}\label{imeas}
\mu^\mathrm{I}_\nu(\{P\}):=\mathcal{I}(s',s''),
\end{equation}
\begin{equation}\label{icmeas}
\mu^\mathrm{IC}_\nu(\{P\}) \;:=\mathcal{I}(s',s'')+\;
\left\{\begin{array}{ll}
|s'|+|s''|-|s'+s''| & \mbox{$i'=i''$},\\
0 & \mbox{$i'\neq i''$}.
\end{array}\right.
\end{equation}
\end{subequations}

The wave strength estimates \eqref{gie} yields balance principles for the wave strength of approximate solutions.
More precisely, given a polygonal region $\Gamma$ with edges transversal to the waves it encounters. Denote by $W^{i\pm}_{\nu,\mathrm{in}}$, $W^{i\pm}_{\nu,\mathrm{out}}$ the positive $(+)$ or negative $(-)$ $i$-waves in $u^\nu$ entering or exiting $\Gamma$, and let $W^i_{\nu,\mathrm{in}}=W^{i+}_{\nu,\mathrm{in}}-W^{i-}_{\nu,\mathrm{in}}$, $W^i_{\nu,\mathrm{out}}=W^{i+}_{\nu,\mathrm{out}}-W^{i-}_{\nu,\mathrm{out}}$. Then the measure of interaction and the measure of interaction-cancelation control the difference between the amount of exiting $i$-waves and the amount of entering $i$-waves w.r.t. the region as follows:
\begin{subequations}
\label{e:bl_bl_1}
\begin{equation*}
\label{e:bl}
|W^i_{\nu,\mathrm{out}}-W^i_{\nu,\mathrm{in}}|\lesssim \mu^\mathrm{I}_\nu(\Gamma),
\end{equation*}
\begin{equation*}
\label{e:bl_1}
|W^{i\pm}_{\nu,\mathrm{out}}-W^{i\pm}_{\nu,\mathrm{in}}|\lesssim \mu^\mathrm{IC}_\nu(\Gamma).
\end{equation*}
\end{subequations}
The above estimates are fairly easy consequences of the interaction estimates \eqref{e:gpe_ve} and the definition of $\mu^\mathrm{I}_\nu$, $\mu^\mathrm{IC}_\nu$. On the other hand, the uniform boundedness of Tot.Var.$\{u(\cdot,t)\}$ w.r.t. time $t$ and parameter $\nu$ implies that $\mu^I_\nu$ and $\mu^{IC}-\nu$ are bounded measures for all $\nu$

By taking a subsequence and using the weak compactness of bounded measures, there exist bounded measures $\mu^{I}$ and $\mu^\mathrm{IC}$ on $\R^+\times\R$ such that the following weak convergence holds:
\begin{equation*}\label{def of muic}
\mu^\mathrm{I}_{\nu}\rightharpoonup\mu^\mathrm{I}, \quad \mu^\mathrm{IC}_\nu\rightharpoonup\mu^\mathrm{IC}.
\end{equation*}

\section{Construction of subdiscontinuity curves}\label{s:sdc}
       Suppose $[\ul,\ur]$ with $\ur=T_i[\ul](s)$ is a wave front of $i$-th family in the approximate solution $u_\nu$, and the wave curve $\tau\mapsto T_i[\ul](\tau)$ (see Theorem \ref{t:ec}) intersects $Z^j_i, \cdots, Z_i^{j+p}$ at $u_j,\cdots,u_{j+p}$ for $0\leq \tau \leq s$, such that
\[
 u_j=T_i[\ul](s_j).
\]
Then we say that the wave front $[u^L,u^R]$ has $(i,k)$-substrength $s_i^k:=s_{k+1}-s_k$ and we decompose the front into $(i,k)$-subdiscontinuity fronts with strength $s_i^k$, where $k\in\{j,\cdots,j+p-1\}\cap2\Z$ when $s>0$, or $k\in\{j,\cdots,j+p-1\}\cap (2\Z+1)$ when $s<0$. The points $u_k$ and $u_{k+1}$ are connected by the part of the curve $T_i[u^L](\cdot)$ inside $\Delta^k_i$.

 We denote the family of all $(i,k)$-subdiscontinuity fronts as $\S_i^k$.

It is obviously that only mixed fronts and discontinuity fronts can have $(i,k)$-substrength $s_i^k>0$ for some $k$, which means they can be decomposed into subdiscontinuity fronts.

\begin{lemma}\label{l:1subdis}
In a wave-front tracking approximation solution, an interaction can only generate at most one subdiscontinuity front with strength $s_i^k$ for some $i,k$.
\end{lemma}
\begin{proof}
By the construction of approximate Riemann solver and the uniformly small total variation of approximate solutions, it is sufficient to prove that the Lipschitz continuous curve $T_i[u^0](\cdot):[0,s]\rightarrow \R^N$ can intersect with $Z_i^j$ at most once for any $u^0\in \Omega$ and all $j$ if $s>0$ is sufficiently small. In fact, by Theorem \ref{t:ec} $\lim_{t\rightarrow 0}T_i[u^0](t)=r_i(u^0)$, one has for $t\in[0,s]$,
\[
 \left|\frac{d}{dt}T_i[u^0](t)-r_i(T_i[u^0](t))\right|\lesssim t.
\]
Recall the assumption $r_i(u)\cdot \mathtt{n}>0$ on $Z^j_i$, one has
\[
 \frac{d}{dt}T_i[u^0](t)\cdot \mathtt{n}>0\ \text{on}\ Z^j_i \quad \text{for $t\in[0,s]$}
\]
as long as $s$ small enough. This concludes that $T_i[u^0](\cdot):[0,s]\rightarrow \R^N$ can intersect $Z_i^j$ at most once.

\end{proof}

Suppose $y_1\in \S_i^{k'}$, $y_2\in \S_i^{k''}$ and $y_1, y_2$ belong to the same wave front, then we say artificially $y_1$ is on the left (right) of $y_2$ if $k',k''$ are even (odd) and $k'<k''\ (k'>k'')$. Then it is easy to see the following lemma.

\begin{lemma}\label{l:non-crossing}
 Suppose two subdiscontinuity fronts $y_1\in \S_i^{k'}$, $y_2\in \S_i^{k''}$ interact and generate two subdiscontinuity fronts $y'_1\in \S_i^{k'}$, $y'_2\in \S_i^{k''}$. Then if $k'<k''\ (k'>k'')$, $y_1$ must be on the left (right) of $y_2$ and also $y'_1$ must be on the left (right) of $y'_2$.
\end{lemma}

Base on these two lemmas, we can follow the idea of \cite{BC} to define approximate subdiscontinuity curves.

\begin{definition}\label{d:asc}
   		 Given $\epsilon\ne 0$, in $u_\nu$, an $(\epsilon,i,k)$-approximate subdiscontinuity curve is a polygonal line in $(x,t)$-plane with nodes $(t_0,x_0),(t_1,x_1),\cdots,(t_n,x_n)$ satisfying
\begin{enumerate}
 \item $(t_j,x_j)$ are interaction points with $0\leq t_0<t_1<\cdots<t_n$.
\item For $1\leq j \leq n$ the segment joining $(t_{j-1},x_{j-1}),\ (t_j,x_j)$ is an $(i,k)$-subdiscontinuity front with $s_i^{k}\geq \epsilon/2$ when $\epsilon>0$ ($s_i^{k}\leq \epsilon/2$ when $\epsilon<0$), and there is at least one index $j'\in \{1,\cdots,n\}$ such that the wave front connecting $(t_{j'-1},x_{j'-1})$, $(t_{j'},x_{j'})$ has $(i,k)$-substrength $s_i^{k}\geq \epsilon$ if $\epsilon>0$ ($s_i^{k}\leq \epsilon$ if $\epsilon<0$).
\item $\forall k<N$, one selects the $(i,k)$-subdiscontinuity fronts with larger speed, that is, if two fronts with substrengths $s_i^k>\epsilon/2$ interact at the node $(x_k,t_k)$, then the front of the $(\epsilon,i,k)$-approximate subdiscontinuity curve is the one coming from the left.
\end{enumerate}
\end{definition}

An $(\epsilon,i,k)$-approximate subdiscontinuity curve which is maximal w.r.t. set inclusion is called a \emph{maximal $(\epsilon,i,k)$-approximate subdiscontinuity curve}.

Let $M^{k}_{i,\nu}(\epsilon)$ be the number of maximal $(\epsilon,i,k)$-approximate subdiscontinuity curves in $u_\nu$.
\begin{lemma}
For fixed $k$ and $\epsilon$, $M^{k}_{i,\nu}(\epsilon)$ is uniformly bounded w.r.t $\nu$.
\end{lemma}

\begin{proof}
We consider the case for $\epsilon>0$ (and the case when $\epsilon<0$ is similar). Since the total variation of $u_\nu(0,\cdot)$ is uniformly bounded by \eqref{initial approx}, the number of maximal $(\epsilon,i,k)$-approximate subdiscontinuity curves which start at time $t=0$ is clearly of order $\epsilon^{-1}$. The number of  $(\epsilon,i,k)$-approximate subdiscontinuity curves which start at a time $t_0>0$ and do not end in finite time is also of order $\epsilon^{-1}$, because the total variation of $u_\nu(\cdot,t)$ is uniformly bounded w.r.t time $t$ and $\nu$.

Now considering an $(\epsilon,i,k)$-approximate subdiscontinuity curve $y_\nu$ which starts at a time $t_0>0$ and ends in finite time, we claim that
\begin{equation}\label{claim_mu}
 \mu^{IC}_\nu(y_\nu)\gtrsim \epsilon^2.
\end{equation}
If the claim is true, as the total amount of interaction and cancelation in the solution $u_\nu$ is uniformly bounded, the number of such $(\epsilon,i,k)$-approximate subdiscontinuity curves is of order $\epsilon^{-2}$.

Thus, combining these situations, we finally obtain the estimate
\[
 M^{k}_{i,\nu}(\epsilon)\lesssim \epsilon^{-2},
\]
which is uniformly valid as $\nu\rightarrow\infty$.

Now we prove the claim. Suppose one $(i,k)$-subdiscontinuity front on $y_\nu$ with $(i,k)$-substrength $\alpha>0$ interact with a front of $j$-th family with strength $\beta^*$ at point $P$, generating an $i$-wave with $(i,k)$-substrength $\gamma\geq 0$ which means that either there is a front of $i$-wave with $(i,k)$-substrength $\gamma>0$ or there is no fronts of $i$-wave having $(i,k)$-substrength. We also assume that  $\gamma<\alpha$. Setting $\theta=\alpha-\gamma$.

First, we consider the case when $i\ne j$, we assume that $i>j$, the case $i<j$ is similar to prove. For notational convenience, we also denote $\alpha, \beta^*, \gamma$ as the front themselves.

Assume that $\alpha$ locates on the front $[\ul,u^M]$ with $u^M=T_i[\ul](\alpha^*)$ and $\ur=T_j[u^M](\beta^*)$,  $\gamma$ locates on the front of the $i$-wave $[\tilde{u}^L,\tilde{u}^R]$ with $\tilde{u}^R=T_i[\tilde{u}^L](\gamma^*)$.

We know that (see section 9.9 of \cite{Daf2})
\begin{subequations}
\label{e:diffwave}
\begin{equation}
\label{e:diffwave_L}
\tilde{u}^L=u^L+s_j \sum_{j<i}r_j(u^M)+O(|s^L|)\alpha^*+o(|s^L|),
\end{equation}
\begin{equation}
\label{e:diffwave_R}
\tilde{u}^R=u^R-s_j \sum_{j>i}r_j(u^M)+O(|s^R|)\beta^*+o(|s^R|),
\end{equation}
\end{subequations}
where
\[
 s^L=(s_1,\cdots,s_{i-1},0,\cdots,0)\quad \text{and} \quad s^R=(0,\cdots,0,s_{i+1},\cdots,s_N).
\]

 Then, by \eqref{gie}, the assumption $i>j$ and $|u^M-u^R|\lesssim |\beta^*|$, we have
\begin{subequations}\label{e:dlr}
\begin{equation}
\label{e:dl}
 |\tilde{u}^L-u^L|\lesssim |\beta^*|+\I(\alpha^*,\beta^*),
\end{equation}
\begin{equation}
\label{e:dr}
 |u^M-\tilde{u}^R|\lesssim |\beta^*|+ \I(\alpha^*,\beta^*).
\end{equation}
\end{subequations}

From Glimm's interaction estimates \eqref{gie}, the parametrization \eqref{e:para_T_k} and the estimates \eqref{e:dlr}, one concludes that the difference of $(i,k)$-substrength between $\alpha$ and $\gamma$ is controlled by the amount of interaction and the strength of the wave $\beta^*$, that is
\begin{equation}
 \theta=\alpha-\gamma \lesssim \I(\alpha^*,\beta^*)+|\beta^*|.
\end{equation}

From definition of $\I$ (see Section \ref{interaction amount}), we know that here $\I(\alpha^*,\beta^*)=|\alpha^*\beta^*|$. So we get $|\beta^*|\gtrsim \theta$ since $|\alpha^*|\ll 1$. Therefore from $|\alpha^*|\geq \epsilon/2$, by the estimates in (i)-(iii), we obtain
\begin{equation*}
  \I(\alpha^*,\beta^*)\gtrsim \epsilon\theta.
\end{equation*}
By notation of interaction measure \eqref{imeas}, one obtains
\begin{equation}\label{ime}
 \mu^I_\nu(P)\gtrsim \epsilon \theta.
\end{equation}

Next we consider the case when $i=j$ and $\alpha$ is on the left of $\beta^*$. (It is similar for the case when $\alpha$ is on the right  of $\beta^*$.)

 First we assume that $\alpha>0,\beta^*<0$. Since the decreasing of the $(i,k)$-substrength can only be caused by interaction and cancellation effects, then similarly one has
\begin{equation*}
  \theta\lesssim |\beta^*|+ \I(\alpha^*,\beta^*).
\end{equation*}

Second, we assume that $\alpha>0,\beta^*>0$.

From the equality \eqref{e:diffwave_L}, one know that
\begin{equation}\label{e:same_wave_L}
 \langle l^0, \tilde u^L-u^L\rangle \lesssim \I(\alpha^*,\beta^*).
\end{equation}
From the equality \eqref{e:diffwave_R},
\begin{equation*}
\begin{split}
&\langle l^0, u^M-\tilde u^R\rangle+\langle l^0, u^M-u^R\rangle  \\
                               =    & \langle  l^0, u^R-\tilde u^R\rangle\lesssim \I(\alpha^*,\beta^*).
\end{split}
\end{equation*}
Since $\langle l^0, u^M-u^R\rangle =-\beta^*<0$, one has
\begin{equation}\label{e:same_wave_R}
\begin{split}
 \langle l^0, u^M-\tilde u^R\rangle \lesssim \I(\alpha^*,\beta^*).
\end{split}
\end{equation}

Noticing that \eqref{e:same_wave_L} and \eqref{e:same_wave_R} implies that the decreasing of the $(i,k)$-substrength can be controlled by the amount of interaction, one has
\[
 \theta \lesssim \I(\alpha^*,\beta^*).
\]

Therefore from the definition of $\mu^{IC}_\nu$ \eqref{icmeas}, in the case when $i=j$ one has
\begin{equation}\label{canem}
 \mu^{IC}_\nu(P)\gtrsim \theta.
\end{equation}

 From (2) of Definition \eqref{d:asc}, the $(i,k)$-substrength all front outgoing from the terminal point of $y_\nu$ must be less than $\epsilon/2$ and there is at least one front on $y_\nu$ has $(i,k)$-substrength larger than $\epsilon$. Then by \eqref{ime} and \eqref{canem}, one can conclude that the claim \eqref{claim_mu} is true.
\end{proof}
\vspace{12pt}
 Up to a subsequence, one can assume that $M^{k}_{i,\nu}(\epsilon) = \bar{M}^{k}_{i}(\epsilon)$ is a constant independent of $\nu$ .
Denote
\begin{equation*}
 y_{m,\nu}^{k,\epsilon}:[t_{m,\nu}^{k,\epsilon-},t_{m,\nu}^{k,\epsilon+}]\rightarrow \R,\qquad m=1,\cdots ,\bar{M}^{k}_{i}(\epsilon)
\end{equation*}
as the maximal  $(\epsilon,i,k)$-approximate subdiscontinuity curves in $\unu$.

Define $\T^k_{i,\nu}(\epsilon)$ as the collection of all maximal $(\epsilon,i,k)$-approximate subdiscontinuity curves in $\unu$ for fixed $\epsilon,i$ and $k$, i.e.
\begin{equation*}
\T^k_{i,\nu}(\epsilon)=\{y_{m,\nu}^{k,\epsilon}: \ m=1,\cdots ,\bar{M}^{k}_{i}(\epsilon)\}.
\end{equation*}

Set
\begin{equation*}
 \mathscr{T}_{i,\nu}:= \bigcup_{k,\epsilon}\mathscr{T}_{i,\nu}^k(\epsilon).
\end{equation*}
as the family of all approximate subdiscontinuity curves of $i$-th family in $u_\nu$.

Up to a diagonal argument and by a suitable labeling of the curves, one can assume that for each fixed $h$, $m$, as $\nu\rightarrow \infty$ the Lipschitz continuous curves
$y_{m,\nu}^{k,\epsilon}$ converge uniformly to some Lipschitz continuous curves $y_{m}^{k,\epsilon}$ which are called \emph{$(\epsilon,i,k)$-subdiscontinuity curves}. Let us denote by
\begin{equation*}
\mathscr{T}_i^k(\epsilon) := \{y_{m}^{k,\epsilon}: \ m=1,\cdots \bar{M}^{k}_{i}(\epsilon)\}
\end{equation*}
the collection of all these limiting curves for fixed $i,k,\epsilon$.

Let
\begin{equation*}
\mathscr{T}_i:=\bigcup_{k,\epsilon}\mathscr{T}_i^k(\epsilon).
\end{equation*}
denote the collection of all these \emph{$i$-subdiscontinuity curves}.

\begin{lemma}
Let $y^{k}_m :]t^-_m,t^+_m[\rightarrow \R$ be an $(\epsilon,i,k)$-subdiscontinuity curve. If $t\in ]t^-_m,t^+_m[$ is such that $(t,y^{k}_m(t))\notin \Theta$, then the derivative $\dot{y}^{k}_m(t)$ exists.
\end{lemma}
\begin{proof}
    There exist $\epsilon_0>0$ and $\gnum \in \T^k_{i,\nu}(\epsilon_0)$, such that
\[
 \gnum\rightarrow \gm,
\]
as $\nu\rightarrow \infty$ and the substrength of $\gnum$ is $|s^{k}_i|\geq \epsilon_0$.

Since in the wave-front tracking approximation, the change of speed of $i$-subdiscontinuity fronts is controlled by the measure $\mu^{IC}_\nu$.

Then for any $\delta_\nu\rightarrow 0$, we deduce that
\begin{equation*}
 \limsup_{\nu\rightarrow \infty}\sup_{|t-t'|<\delta_\nu}|\dot{y}^{k}_{m,\nu}(t)-\dot{y}^{k}_{m,\nu} (t')|=0.
\end{equation*}
From uniformly convergence $\gnum\rightarrow \gm$ on a neighborhood of $t$, we obtain
\begin{equation}\label{subdis.-conv.}
  \dot{y}_m^{k}(t)=\lim_{\nu\rightarrow \infty} \dot{y}^{k}_{m,\nu}.
\end{equation}
\end{proof}

Recall the definition of generalized characteristics which will be used in the proof of Theorem \ref{t:main theorem}.
\begin{definition}
 A \emph{generalized $i$-characteristic} associated with the approximate solution $u_\nu$, on the time interval $[t_1,t_2]\subset [0,\infty)$, is a Lipschitz continuous function $\chi:[t_1,t_2]\rightarrow (-\infty,\infty)$ which satisfies the differential conclusion
\[
 \dot{\chi}(t)\in [\lambda_i(u_\nu(x+,t)),\lambda_i(u_\nu(x-,t))].
\]
\end{definition}

For any given $(T,\bar{x})\in \R$, we consider the \emph{minimal (maximal) generalized $i$-characteristic} through $(T,\bar{x})$ defined as
\[
 \chi^{-(+)}(t)=\min(\max)\{\chi(t):\chi \text{ is a generalized characteristic, }\chi(T)=\bar{x}\}.
\]

The properties of approximate solutions yield that there is no wave-front of $i$-th family crossing $\chi^+$ from the left or crossing $\chi^-$ from the right.

Suppose $\T_{i,\nu} \ni y'_\nu:[t'^-_\nu,t'^+_\nu]\rightarrow \R$ and $ \T_{i,\nu} \ni y''_\nu(\cdot):[t''^-_\nu,t''^+_\nu]\rightarrow \R$. By Lemma \ref{l:non-crossing}  and the definition of $(\epsilon,i,k)$-approximate subdiscontinuity curves, it turns out that either
\[
 y'_\nu(t)\leq \ y''_\nu(t),\quad \forall t\in [t'^-_\nu,t'^+_\nu]\cap [t''^-_\nu,t''^+_\nu],
\]
or
\[
 y'_\nu(t)\geq \ y''_\nu(t), \quad \forall t\in [t'^-_\nu,t'^+_\nu]\cap [t''^-_\nu,t''^+_\nu].
\]
This makes the following definition well defined.
\begin{definition}\label{shock order}
Suppose $y'\in \T^{k'}_{i}(\epsilon'), \ y'' \in \T^{k''}_{i}(\epsilon'') $ where $k'$ and $k''$ are both odd numbers (or even numbers), and $y'_\nu\rightarrow y',\ y''_\nu \rightarrow y''$ as $\nu\rightarrow \infty$ and assume there exists a point $(t_0,x_0)$ with $t_0>0$ such that $y'(t_0)=y''(t_0)=x_0$, we say $y'\prec y''$ if there exists a neighborhood $[t^-,t^+]$ of time $t_0$ such that
\begin{itemize}
 \item $y'(t) \leq y''(t)$ for all $t \in [t^-,t^+]$,
 \item either there exist $t^* \in [t^-,t^+]$ such that $y'(t^*) < y''(t^*)$

       or  for all $t \in [t^-,t^+]$, $y'(t)=y'(t)$, and $k_1<k_2\ (k_1>k_2)$.
\end{itemize}
\end{definition}

The next lemma rules out the situation when two subdiscontinuity with different sign of strength is tangent at the same point which is not the atom point of interaction and cancellation measure.
\begin{lemma}\label{l:int of two subcurves}
Suppose that $y_1\in\T^{k_1}_i,\ y_2\in\T^{k_2}_i$ with $k_1$ is even and $k_2$ is odd, and assume that $y_1$ and $y_2$ pass through the same point $(x_0,t_0)$ with $\mu^{IC}(\{(x_0,t_0)\})=0$ and there is no subdiscontinuity curve $y_0$ such that
\[
 y_1(t)\leq y_0(t)\leq y_1(t)
\]
for any neigbourhood of $t_0$. Then $\dot{y}_1(t_0)\ne \dot{y}_2(t_0)$.
\end{lemma}
\begin{proof}
Suppose $\T^{k_1}_{i,\nu}\ni y_{1,\nu}\to y_1,\ \T^{k_2}_{i,\nu}\ni y_{2,\nu}\to y_2$ and $t_{1,\nu},t_{2,\nu}\to t_0$. Let us denote the points $y_{1,\nu}(t_{1,\nu}),\ y_{2,\nu}(t_{2,\nu})$ as $A_\nu, B_\nu$ respectively. Since there is no subdiscontinuity curves between $y_1$ and $y_2$, the strengths of all fronts crossing the segment $\overline{A_\nu B_\nu}$ tend to zero.

 Moreover, the total strength of the fronts of the other family tends to zero. In fact, if not, either they are canceled in the neighborhood of $(x_0,t_0)$ or interacted with $y_{1,\nu}(t_{1,\nu}),\ y_{2,\nu}(t_{2,\nu})$, which implies the uniform positivity of $\mu_\nu^{IC}$ on a small region $\Gamma_\nu$. This contradicts the assumption that $\mu^{IC}(\{(x_0,t_0)\})=0$. Therefore, the values of each $u_\nu$ along the segment $\overline{A_\nu B_\nu}$ remain arbitrary close to $i$-rarefaction curve.

On the other hand, {\bf since the sign of the strengths of $y_1$ and $y_2$ are different}, one can always find $A'_\nu,\ B'_\nu$ on $\overline{A_\nu B_\nu}$ and a positive constant $c$ such that $u_\nu(A'_{\nu}),u_\nu(B'_{\nu})\in \Delta^k_i$ for some $k$ and $|u_\nu(A'_{\nu})-u_\nu(B'_{\nu})|>c$, therefore

\[
 |\lambda_i(u_\nu(A'_\nu))-\lambda_i(u_\nu(B'_\nu))|\gtrsim c.
\]

Up to subsequence, we can assume that for all $\nu$
\begin{subequations}
\label{char.diff}
\begin{equation}
\label{char.diff-1}
 \lambda_i(u_\nu(A'_\nu))-\lambda_i(u_\nu(B'_\nu))>c/2,
\end{equation}
\begin{equation}
\label{char.diff-2}
\text{or } \lambda_i(u_\nu(A'_\nu))-\lambda_i(u_\nu(B'_\nu))<c/2.
\end{equation}
\end{subequations}

Let us consider the case \eqref{char.diff-1}, the other case is analogous to prove. We take $\chi^+$ through $A'_\nu$, $\chi^-$ through $B'_\nu$. Since $A'_\nu$ and $B'_\nu$ are in the same $\Delta_i^k$, if no uniformly large interaction occur on $\chi^+,\ \chi^-$, they will interact with each other.  We consider the region $\Gamma_\nu$ bounded by $\overline{A_\nu B_\nu},\ \chi^+$ and $\chi^-$. Since no fronts can leave $\Gamma_\nu$ through $\chi^+$ and $\chi^-$. By \eqref{d:interation amount ijwaves} and \eqref{ICmeas}, we obtain that $\mu^{I}(\Gamma_\nu)\gtrsim c$ which contradicts the assumption $\mu^{IC}(\{(x_0,t_0)\})=0$.
\end{proof}

\section{Proof of Theorem \ref{t:main theorem}}\label{s:pf}
Before proving the theorem, we recall the definition space-like curve.

\begin{definition}
 Let $\hat{\lambda}$ be a constant larger than the absolute value of all characteristic speed. We say a curve $x=y(t),\ t\in[a,b]$ is \emph{space-like}  if
\[
 |y(t_2)-y(t_1)|>\hat{\lambda}(t_2-t_1)\quad {\rm for\ all}\ a<t_1<t_2<b.
\]

\end{definition}
From the definition one knows that any fronts can cross a space-like curve at most once.

\begin{proof}[Proof of Theorem \ref{t:main theorem}]
Let $\Theta$ consists of all jump points of initial data, the atom points of interaction and cancelation measure $\mu^{IC}$ and the points where two sub-discontinuity curves of different families cross each other.

Consider a point $P=(\tau,\xi)\notin \Theta$. Since $u(\cdot,\tau)$ has bounded variation, there exist the limits
\[
 u^L:=\lim_{x\rightarrow \xi-}u(x,\tau),\qquad u^R:=\lim_{x\rightarrow \xi+}u(x,\tau).
\]
Assuming that $\ur=T_i[\ul](s)$. We only consider the case for $s>0$ and the case for $s<0$ is analogous to prove.

  Applying the tame oscillation condition (see p.295 of \cite{BB}), one obtains
\begin{equation}\label{e:lim in triangle}
\lim_{\genfrac{}{}{0pt}{}{(x,t)\rightarrow(\xi,\tau)}{\tau\leq t <\tau+(\xi-x)/\hat{\lambda}} }u(x,t) = u^L,\quad \lim_{\genfrac{}{}{0pt}{}{(x,t)\rightarrow(\xi,\tau)}{\tau\leq t <\tau+(x-\xi)/\hat{\lambda}} }u(x,t) = u^R.
\end{equation}
for some constant $\hat \lambda$ which is larger than all characteristic speeds.

Suppose that there are $i$-subdiscontinuity curves $y_1^{k_1},\cdots,y_l^{k_l}$ satisfying  $y_j^{k_j}\in \T^{k_j}_i(\epsilon_j)$ with $\epsilon_j>0,\ 1\leq j\leq l$,
\[
y^{k_1}_1(\tau)=\cdots=y^{k_l}_l(\tau)=\xi
\]
and there is no other subdiscontinuity curve passing through the point $P$. This can be done because of the conclusion of Lemma \ref{l:int of two subcurves} and the fact that $P\notin \Theta$. It is easy to show that $\epsilon_j$ must have the same sign, i.e. $\epsilon_{j_1}\epsilon_{j_2}>0$ for any $j_1,j_2\in\{1,\cdots,l\}$.

By rearranging the index, we can assume that
\begin{equation*}
 y^{k_1}_1\prec y^{k_2}_2\prec \cdots \prec y^{k_l}_l,
\end{equation*}
 and

\begin{itemize}
 \item[(H)]    there is no $y^{k_0}_0 \in \T_i$ such that $ y^{k_0}_0 \prec y^{k_1}_1$ with $y^{k_0}_0(t_0)=y^{k_1}_1(t_0)$ or $y^{k_{l}}_{l}\prec y^{k_{0}}_{0}$ with $y^{k_{0}}_{0}(t_0)=y^{k_{l}}_{l}(t_0)$.

\end{itemize}

{\bf Step 1.}  By the definition, there exist $y_{1,\nu}^{k_1},\cdots, y_{l,\nu}^{k_l}\in \T_{i,\nu}$ such that $y_{j,\nu}^{k_j}\rightarrow y_j^{k_j},\ \forall j\in\{1,\cdots,l\}$. We claim that
\begin{subequations}
\label{e:left_right_lim}
\begin{equation}
\label{e:leftlim}
\lim_{r \rightarrow 0+} \limsup_{\nu \rightarrow \infty} \left( \sup_{\genfrac{}{}{0pt}{}{x<y_{1,\nu}^k(t)}{(x,t)\in B(P,r)}} \big| u_\nu(x,t) - u^L \big| \right) = 0,
\end{equation}
\begin{equation}
\label{e:rightlim}
\lim_{r \rightarrow 0+} \limsup_{\nu \rightarrow \infty} \left( \sup_{\genfrac{}{}{0pt}{}{x>y_{l,\nu}^{k_l}(t)} {(x,t)\in B(P,r)}} \big| u_\nu(x,t) - u^R \big| \right) = 0,
\end{equation}
\end{subequations}
where $B(P,r)$ is a ball centred at $P$ with radius $r$.

Indeed, if \eqref{e:leftlim} is not true , by the first limit in \eqref{e:lim in triangle} and $u_\nu \rightarrow u$ pointwise a.e., there exist two sequences of points $P_\nu,\ Q_\nu$ converging to $ P$ and $P_\nu,\ Q_\nu$ on the left of $y^\nu_1$ such that the segment $\overline{\pnu \qnu}$ is space-like and
\[
u(P_\nu)\rightarrow u^L
\]
and
\[
 |u_\nu(P_\nu)-\unu(Q_\nu)|\geq \epsilon_0.
\]
It is not restrictive to assume that the direction $\overrightarrow{P_\nu Q_\nu}$ towards $\gnuone$.

Let $\Lambda_j(\pq)$ be the total wave strength of wave-fronts of $j$-th family which cross the segment $\pq$. Then, one has $\Lambda_j(\pq)\gtrsim \epsilon_0$ for some $j\in\{1,\cdots,d\}$. We consider the following three cases.

{\bf Case 1.} If $j>i$, we take the maximal forward generalized $j$-characteristic $\chi^+$ through $P_\nu$ and minimal generalized $j$-characteristic $\chi^-$ through $Q_\nu$.

If $\chi^+$ and $\chi^-$ interact each other at $O_\nu$ before hitting $\gnuone$. We consider the region $\Gamma_\nu$ bounded by $\pq,\ \chi^+$ and $\chi^-$. Since no fronts can leave $\Gamma_\nu$ through $\chi^+$ and $\chi^-$. By \eqref{d:interation amount ijwaves} and \eqref{ICmeas}, we obtain that the amount of interaction and cancellation $\muicnu(\bar \Gamma_\nu)$ for $u_\nu$ within the closure of $\Gamma_\nu$ remains uniformly positive as $\nu \to \infty$.

If $\chi^+$ interacts $\gnuone$ at $A_\nu$ and $\chi^-$ interacts $\gnuone$ at $B_\nu$, we consider the region $\Gamma_\nu$ bounded by $\pq,\ \chi^+,\ \chi^-$ and $\gnuone$. Then either there exists a constant $0<c'_0<1$ such that $\muicnu(\Gamma_\nu)>c'_0 \epsilon_0$ or there exists a constant $0<c''_0<1$ such that fronts with total strength lager than $c''_0 \epsilon_0$ hitting $\overline{A_\nu B_\nu}$. By \eqref{d:interation amount ijwaves} and \eqref{ICmeas} we can determine that $\muicnu(\bar{\Gamma}_\nu)\gtrsim \epsilon$ uniformly.

For both above two cases, $\Gamma_\nu$ is contained in a ball $B(P,\r_\nu)$ with $\r_\nu\to 0$ as $\nu\to \infty$, which implies that $\muic(\{P\})>0$. This is against the assumption $P\notin \Theta$.

{\bf Case 2.} If $j<i$, we consider the minimal backward generalized $j$-characteristic through the point $P_\nu$ and the maximal backward generalized $j$-characteristic through the point $Q_\nu$. Then by the similar argument for the case $j>i$, we get $\muic(\{P\})>0$ against the assumptions.

{\bf Case 3.}  If $j=i$ and for any $ j'\ne i,\ 1\leq j' \leq d,\ \Lambda_{j'}(\pq)\rightarrow 0$ as $\nu \rightarrow \infty$.
We claim that the maximum of the strengths of all fronts which cross $\pq$  tends to zero when $\nu \rightarrow \infty$.

If it is not true, since $\Lambda_{j'}(\pq)\rightarrow 0$ for $j'\ne i$, there must be fronts of $i$-th family across $\pq$ with uniformly large strength  We assume that up to a subsequence, their $(i,k_0)$-substrength $s^{k_0}_\nu$ are uniformly large for some $k_0\in \{1,\cdots,\bar{k}_i\}$., that is $|s^{k_0}_\nu|\gtrsim \epsilon$ for some $\epsilon>0$.

Then by Definition \ref{d:asc} there must be, for some $\epsilon_0>0$, $(\epsilon_0,i,k_0)$-approximate subdiscontinuity curves $y^{k_0}_{0,\nu}$ which contains the wave fronts $s^{k_0}_\nu$ (we use the same notation as their strength for convenience) and since $y^{k_0}_{0,\nu}$ are uniformly Lipschitz continuous curves, up to a subsequence, there is a Lipschitz continuous curves $y^{k_0}_0$, such that
\[
 y^{k_0}_{0,\nu} \rightarrow y^{k_0}_0\in \T^{k_0}_i,\quad \nu \rightarrow \infty
\]
and $y^{k_0}_0(\tau)=\xi$. By Definition \ref{shock order} we obtain $y^{k_0}_0\prec y^{k_1}_1$, which contradicts the assumption (H).

So we can always choose $\qnu',\ \pnu'\in \overline{\pnu \qnu}$ such that
\[
 \unu(\qnu')\rightarrow u^L
\]
and
\[
 |\unu(\qnu')-\unu(\pnu')|\geq c_0\epsilon_0,
\]
 where $0<c_0<1$ and $\unu(\qnu'),\ \unu(\pnu')$ locate in the same $\Delta_i^k$ for some $k$.

Since for $j'\ne i,\ \Lambda_{j'}(\pq)$ is arbitrary small when $\nu$ is large enough and the strength of fronts belonging to $i$-th family is small, one has
\[
 \max_{H_\nu\in \overline{P'_\nu Q'_\nu}}\min_{s^*\geq 0}|u_\nu(H_\nu)-R_i[u_\nu(P'_\nu)](s^*)|\ll 1.
\]
Which means the values of each $u_\nu$ along the segment $\overline{P'_\nu Q'_\nu}$ remain arbitrary close to the $i$-rarefaction curve through $u^L$. Then by the analogous argument in the proof of Lemma \ref{l:int of two subcurves}, one gets the contradiction $\muic(\{P\})>0$.
Therefore, we conclude that \eqref{e:left_right_lim} is true. And \eqref{e:rightlim} is similar to prove.

{\bf Step 2.} Define $\T=\bigcup_i \T_i$. If $P\notin \Theta\cap \mathrm{Graph}(\T)$ and if $u$ is not continuous at $P$, then there exist $\epsilon>0$ and $\pnu,\ Q_\nu\rightarrow P$ such that $\pq$ is space like and
\begin{equation*}
 \unu(\pnu)\rightarrow u(P), \qquad |\unu(\qnu)-u(P)|\geq \epsilon \text{ for all }\nu.
\end{equation*}
Up to subsequence, we consider th following two cases.
\begin{itemize}
 \item[1)] there exists $j\ne j'$, such that $\min\{\Lambda_j(\pq),\ \Lambda_{j'}(\pq)\} \geq \epsilon_0$. This situation can be ruled out by the argument in Case 1 of Step 1.
 \item[2)]  For some $j\in\{1,\cdots,j\}$ and all $\nu$, $\Lambda_j(\pq) \geq\epsilon_0$ and for all $ j'\ne j,\ \Lambda_{j'}(\pq)\rightarrow 0$ as $\nu\rightarrow \infty$. Then one can use the argument in Case 3 of Step 1 to obtain the contradiction.
\end{itemize}
Therefore we get the continuity of $u$ outside $\mathrm{Graph}(\T)\bigcup \Theta$.

Then, we have proved the first part of the theorem.

{\bf Step 3.} We now establish the Rankine-Hugoniot condition \eqref{d:R-H condition} for curves in $\T$. Let $P=(t_0,x_0)\in Graph(\T)\setminus\Theta$, and write
\begin{equation*}
\ul=\lim_{x\rightarrow x_0-} u(x,t_0),\qquad \ur=\lim_{x\rightarrow x_0+} u(x,t_0).
\end{equation*}
We consider two cases.

{\bf Case 1.} There is only one curve $y \in \T_i$ passing through point $P$. From \eqref{e:left_right_lim}, we know that the discontinuity $[\ul,\ur]$ must be simple. Suppose that $\T_{i,\nu} \ni y_\nu \rightarrow y$ as $\nu\rightarrow \infty.$ By \eqref{R_H} and \eqref{front speed error}, we obtain
\begin{equation*}
\begin{split}
 \sigma_i(\ul_\nu,\ur_\nu)[\ul_\nu-\ur_\nu]=f(\ulnu)-f(\urnu),\\
|\dot{y}_\nu-\sigma_i(\ulnu,\urnu)|<2\epsilon_\nu,
\end{split}
\end{equation*}
where
\begin{equation*}
\ul_\nu=\lim_{x\rightarrow y_\nu(t_\nu)-} u(t_\nu, x),\qquad \ur_\nu=\lim_{x\rightarrow y_\nu(t_\nu)+} u(t_\nu, x), \quad t_\nu\rightarrow t_0 \ \text{as}\ \nu\rightarrow \infty.
\end{equation*}

Then by \eqref{e:left_right_lim}, one has for every $\epsilon>0$ there exists $\bar{\nu}(\epsilon)$ such that $\forall \nu>\bar{\nu}$, one has
\begin{equation*}
 |y_\nu-\sigma_i(\ul,\ur)|\leq |y_\nu-\sigma_i(\ulnu,\urnu)|+|\sigma_i(\ul,\ur)-\sigma_i(\ulnu,\urnu)|<\epsilon.
\end{equation*}
From \eqref{subdis.-conv.} and the fact that the $i$-waves of Riemann problem constructed by Theorem \ref{t:ec} are Liu admissible, one deduces that
\[
 \dot{y}(t_0)(u^R-u^L)=f(u^R)-f(u^L),
\]
and
\begin{equation*}
 \dot{y}(t_0)\leq \hat \sigma_i(S_i[\ul](\tau),\ul),\ \forall \tau\in [0,s].
\end{equation*}

{\bf Case 2.} If the discontinuity $[\ul,\ur]$ is a composition of $[u_0,u_1],\ [u_1,u_2],\cdots,[u_l,u_{l+1}]$ where $u_0=\ul,u_{l+1}=\ur$ , $u_j=T_i[\ul](s_j)$ and $\ul\in\Delta^{k^*_1}_i,\ur\in\Delta^{k^*_2}_i$. Let
\begin{align*}
  k_1&=\min\{k\ \text{is even},\ k\geq k^*_1\},\\
 k_{p}&=\max\{k\ \text{is even},\ k\leq k^*_2\}
 \end{align*}
and $k_p=k_1+2(p-1)$. One has $p\geq l+1$.

According to \eqref{e:left_right_lim}, there exist $P_\nu,Q_\nu\rightarrow P$ such that
\[
 u_\nu(P_\nu)\rightarrow \ul,\ u_\nu(Q_\nu)\rightarrow \ur,
\]
and the segment $\pq$ is space-like.

  We claim that there exist $p$ subdiscontinuity curves $y_1,\cdots,y_{p} \in \T_i$ passing through point $P$, where $y_j\in\T^{k_1+2j}_i,\ j=1,\cdots,p-1.$

In fact, let $\S^k_i(\pq)$ denote the maximal $(i,k)$-substrength of all fronts across $\pq$. It is sufficient to show that, up to a subsequence, there is a constant $C>0$ such that
\[
 \S^k_i(\pq)\geq C,\ \forall k\in\{k_1,\cdots,k_{p}\}.
\]

If not, there exists $k_0\in \{k_1,\cdots,k_{p}\}$ such that $\S^{k_0}_i(\pq)\rightarrow 0$ as $\nu\rightarrow \infty$. Since $\mu^{IC}(P)=0$, we have
\[
 \Lambda_j(\pq)\rightarrow 0,\ \text{as}\  \nu\rightarrow \infty,\  \forall j\ne i.
\]
By Lemma \ref{l:int of two subcurves}, we conclude that $\S^k_i(\pq)\rightarrow 0,\ \text{as}\  \nu\rightarrow \infty,\ \forall k$ odd. Since the $(i,k_0)$-substrength of all wave fronts are arbitrary small, up to a subsequence, one can always find points $P'_\nu,Q'_\nu$ on $\pq$ such that $u_\nu(P'_\nu),\ u_\nu(Q'_\nu)\in \Delta^{k_0}_i$ and all fronts are either admissible discontinuity with left and right value inside $\Delta^{k_0}_i$ or rarefaction fronts. Therefore by the analogous argument in Case 3 of Step 1, there exist a constant $c>0$ independent on $\nu$ such that in a small neighborhood $\Gamma_\nu$ of $P$, one has $\mu^{IC}_\nu(\Gamma_\nu)\geq c$. This contradicts with the assumption $\mu^{IC}(P)=0$. This concludes our claim.

Moreover, by Lemma \ref{l:int of two subcurves} and the equalities \eqref{e:left_right_lim}, there are exactly $p$ subdiscontinuity curves passing through $P$, otherwise $P$ must be an interaction point with $\mu^{IC}(P)>0$.

Suppose $u_{n+1}=T_i[u_n](s_n)$ for some $s_n>0$, $n\in\{0,\cdots,l\}$ and $T_i[u_n](\cdot)$ intersects $Z_i^{k_{n_1}},\cdots,Z_i^{k_{n_q}}$ at $u_{n,1},\cdots,u_{n,q}$. Then the subdiscontinuity curves with substrength $s^{k_{n_1}}_i,\cdots,s^{k_{n_q}}_i$ must conincide in neighborhood of time $\tau$.

If in a neigborhood of $P$, $y_j$ and $y_{j+1}$ are not identical, by the similar argument for proving \eqref{e:left_right_lim}, one can show that \eqref{e:lim inside} is true.

Suppose $\T_{i,\nu} \ni y_{j,\nu}\rightarrow y_{j},\ j\in \{1,\cdots,p\}$. As we discussed above, it is not restrictive to assume that $p=l+1$ and there is a neighborhood $U(t_0)$ of $t_0$, such that $\forall t\in U(t_0)$,
\[
 y_{1,\nu}(t)<\cdots<y_{p,\nu}.
\]

Then, using the similar argument in Step 1, besides \eqref{e:left_right_lim} and \eqref{e:rightlim}, one can also show that

\begin{equation*}
 \begin{split}
 & \lim_{r \rightarrow 0+} \limsup_{\nu \rightarrow \infty} \left( \sup_{\genfrac{}{}{0pt}{}{x<y_{m,\nu}^{k_m}(t)}{(x,t)\in B(P,r)}} \big| u_\nu(x,t) - u_{m-1} \big| \right) = 0 \ m=2,\cdots, p,\\
&\lim_{r \rightarrow 0+} \limsup_{\nu \rightarrow \infty} \left( \sup_{\genfrac{}{}{0pt}{}{x>y_{n,\nu}^{k_n}(t)}{(x,t)\in B(P,r)}} \big| u_\nu(x,t) - u_{n} \big| \right) = 0,\ n=1,\cdots,l.
 \end{split}
\end{equation*}

For notational convenience, we write $u_0=u^L,\ u_{p}=u^R$. Therefore, by the same argument in Case 1, we obtain that for $n=1,\cdots,p$
\[
 \dot{y}_n(t_0)(u_n-u_{n-1})=f(u_n)-f(u_{n-1})
\]
and
\[
  \dot{y}_n(t_0)\leq \hat \sigma_i(u_n,S_i[\ul](\tau)),\ \forall \tau\in [s_{j-1},s_{j}].
\]

Adding them together, one finally obtain for $m=1,\cdots,p$
\[
  \dot{y}_m(t_0)(u^R-u^L)=f(u^R)-f(u^L)
\]
and
\begin{equation*}
 \dot{y}_m(t_0)\leq \hat \sigma_i(\ul,S_i[\ul](\tau)),\quad \forall \tau\in [0,s].
\end{equation*}

\end{proof}

The proof of Theorem \ref{t:main theorem} already implies the results of Theorem \ref{t:approx_shock_conv}.

Concerning a sequence of exact solution of \eqref{basic equation} such that $u_\nu\rightarrow u$ in $L^1_{loc}$, one can approximate each $u_\nu$ by a sequence of wave-front tracking approximations $u_{m,\nu}\rightarrow u_\nu$ and by a suitable diagonal sequence $u_{\nu,m(\nu)}\to u$ one has the following corollary.
\begin{corollary}[stability of discontinuity curves]
 Considering a sequence of exact solutions $u_\nu$ such that $u_\nu\rightarrow u$ in $L^1_{loc}$, one has
\begin{enumerate}
 \item Let $y_\nu:[t^-_\nu,t^+_\nu]\mapsto \R$ be a discontinuity curves of $u_\nu$ described in Theorem \ref{t:main theorem}. Assume $t^-_\nu\to t^-,t^+_\nu\to t^+$ and $y_\nu(t)\to y(t)$ for each $t\in[t^-,t^+]$. Then $y(\cdot)$ is a discontinuity curve of the limiting solution $u$ with the properties mentioned in Theorem \ref{t:main theorem}.
\item Viceversa, let $y:[t^-,t+]\mapsto \R$ be a discontinuity curve of $u$ for a.e. $t\in[t^-,t^+]$. Then there exists a sequence of discontinuity curves $y_\nu:[t^-_\nu,t^+_\nu]\mapsto \R$ of $u_\nu$ such that
\[
 t^-_\nu\to t^-,\ t^+_\nu\to t^+,\quad \lim_{\nu\to \infty}y_\nu(t)=y(t),
\]
for a.e. $t\in[t^-,t^+]$.
\end{enumerate}

\end{corollary}

\section{An remark on general strict hyperbolic systems}
\label{s:example}

We construct a strict hyperbolic  system of conservation laws with one characteristic family which is not linearly degenerate or piecewise genuinely nonlinear. Therefore neither the assumption of Theorem \ref{t:main theorem} or that of Theorem 10.4 in \cite{Bre} holds. We show that the set of jump points of its admissible solution to some initial data can not be ``exactly'' covered by countably many Lipschitz continuous curves.

Consider the following $2\times 2$ system
\begin{equation}
\label{e:example}
\begin{cases}
  u_t+f(u,v)_x=0,\\
v_t-v_x=0.
\end{cases}
\end{equation}
where $f$ is a smooth function and $u,v$ is the unknown variables.
The Jacobian matrix of flux function is

\[
 DF(u,v) =
 \begin{pmatrix}
  f_u & f_v \\
 0 & -1
 \end{pmatrix}.
\]
Then the eigenvalues are
\[
 \lambda_1=-1,\quad \lambda_2=f_u.
\]
And the corresponding right eigenvectors are
\[
 r_1(u,v)=(f_v,-f_u-1)^T,\quad r_2=(1,0)^T.
\]

The system is strict hyperbolic if $f_u>-1$. (In fact, $f$ constructed latter satisfies this property). Obviously, one has
\[
Z_1=\{(u,v):\nabla \lambda_1\cdot r_1(u,v)=0\}=\R^2,
\]
which means that the 1-th characteristic family is linearly degenerate.

Latter we will also show that
\begin{equation}\label{e:inflection_manifold_2}
 Z_2=\{(u,v):\nabla \lambda_2\cdot r_2\}=\{(u,v):f_{uu}(u,v)=0\}=\{v=0\}.
\end{equation}
This yields that the vector field $r_2$ is tangent to the manifold $Z_2$, therefore the second characteristic family is not piecewise genuinely nonlinear or linearly degenerate..

Define $f(u,v)=e^{-1/v}u^2/{2}$ when $v>0$ and $f(u,0)\equiv 0$. In the following, we define the value of $f$ for $v<0$.

Let the initial data to be
\begin{equation}
 u_0(x)=\begin{cases}
         u_l & x<0,\\
         u_r & x>0,
        \end{cases}
\qquad v_0(x)=\begin{cases}
        -a & x<h,\\
       a & x>h,\\
          \end{cases}
\end{equation}
for some small constants $u_l>u_r$ and $a,h>0$.

Since the second equation in \eqref{e:example} is a linear transport equation, one has
\begin{equation}
v(x,t)=\begin{cases}
        -a & x+t<h,\\
       a & x+t>h.
           \end{cases}
\end{equation}
Then one can solve the system \eqref{e:example} by regarding it as a scalar conservation laws of $u$
\[
  u_t+f(u,v)_x=0
\]
with discontinuous coefficient when $v$ is a piecewise constant function.

If w.r.t. $u$, $f$ is concave for some small fixed $v<0$ and concave for some fixed $v>0$, then $u$ is a center rarefaction wave in the area $\{x+t<h\}$. Immediately after  the rarefaction wave crosses the characteristic line $x+t=h$, it turns out to be a compressive wave since $f(u,-a)$ is a convex function on $u$. It may generate a shock after a short time.

Indeed, this can be done in the following way.

Consider a center rarefaction wave of $u$ in the area $\{x+t<h\}$
\begin{equation}
 u(x,t)=\begin{cases}
  u_l & x/t< f_u(u_l,-a),\\
  g(x/t) & x/t= f_u(g(x/t),-a),\\
  u_r & x/t> f_u(u_r,-a),
 \end{cases}
\end{equation}
where $g(\cdot)$ is the converse function of $f_u(\cdot,-a)$.

Let $u^-= g(\xi^*)$ for some $\xi^*$. The value of $u$ will be $u^-$ along the characteristic line
\[
 x=t f_u(u^-,-a)
\]
until it intersects another characteristic line $x+t=h$.

Solving the equations
\begin{equation}
 \begin{cases}
  x=t f_u(u^-,-a),\\
  x=-t+h,
 \end{cases}
\end{equation}
we get the intersection point of two characteristics:
\begin{equation}
 \begin{cases}
  x_0=h f_u(u^-,-a)/[1+f_u(u^-,-a)],\\
  t_0={h}/[1+f_u(u^-,-a)].
 \end{cases}
\end{equation}

Next, we compute the value of $u$ after the characteristic cross the line ${x+t=h}$. For each point $(x_0,t_0)$ on the line $x+t=h$, let
\[
 u^+(x_0,t_0)=\lim_{\substack{x+t>h\\ (x,t)\to(x_0,t_0)}}u(x,t).
\]
By Rankine-Hugoniot condition, one has
\[
 -(u^+-u^-)=\frac{e^{-1/a}(u^+)^2}{2}-f(u^-,-a).
\]
This yields
\begin{equation}
\label{e:value_u+}
 u^+=\frac{-1+\sqrt{1+2e^{-1/a}(f(u^-,-a)+u^-)}}{e^{-1/a}}.
\end{equation}

Since $f(u,v)={e^{-1/a}u^2}/{2}$ on area $\{h<x+t<2h\}$ , the characteristic line of the equation
\[
 u_t+f(u,a)_x=0
\]
starting at $(x_0,t_0)$ is
\[
x-\frac{h f_u(u^-,-a)}{1+f_u(u^-,-a)}=e^{-1/a} u^+ (t-\frac{h}{1+f_u(u^-,-a)}).
\]

We require that it passes through the point $(0,2h)$, that is
\[
 -\frac{h f_u(u^-,-a)}{1+f_u(u^-,-a)}=e^{-1/a} u^+ (2h-\frac{h}{1+f_u(u^-,-a)}),
\]
which yields
\begin{equation} \label{e:value_fu}
 f_u(u^-,-a)=\frac{-e^{-1/a}u^+}{2e^{-1/a }u^+ +1}.
\end{equation}

Substitute \eqref{e:value_u+} into \eqref{e:value_fu}, we get
\begin{equation*}
 f_u(u^-,-a)=\frac{1-g(u^-,-a)}{2g(u^-,-a)-1},
\end{equation*}
where $g(u^-,-a)=\sqrt{1+2e^{-1/a}(f(u^-,-a)+u^-)}.$

Now, we consider the Cauchy problem of an ODE with parameter $a$
\begin{equation}
\label{e:ode}
 \begin{cases}
  \frac{d}{du}F(u,a)=\frac{1-G(F,u,a)}{2G(Fu,a)-1},\\
  F(0,a)=0,
 \end{cases}
\end{equation}
where $G(F,u,a)=\sqrt{1+2e^{-1/a}(F+u)}$.

By the theory of the ODE, since $G$ is smooth when $(u,a,F)$ lies in a small neighborhood of the origin, there is a unique smooth solution $F$ defined on some interval $[-b,b]\ (b>0)$, smoothly depending on the parameter $a$.

Therefore we defined $f$ for $v<0$ small as
\[
 f(u,v)=F(u,v).
\]

By our construction, $f(u,v)$ is concave about $u$ for a negative fixed $v$. In fact, from \eqref{e:ode}, one finds
\[
 f_{uu}=\frac{-g_u}{2 g+1}<0,
\]
since $g_u=\frac{2e^{-1/v}}{2h}>0$. This concludes that $f$ is concave with respect to $u$.

It is not difficult to verify that all partial derivatives of $f$ is continuous on $\{v=0\}$.  As $f$ is smooth on $\{v\ne 0\}$, it concludes that $f$ is smooth and independent of $h$.

\begin{figure}[htbp]
\begin{center}

\begin{picture}(0,0)%
\includegraphics{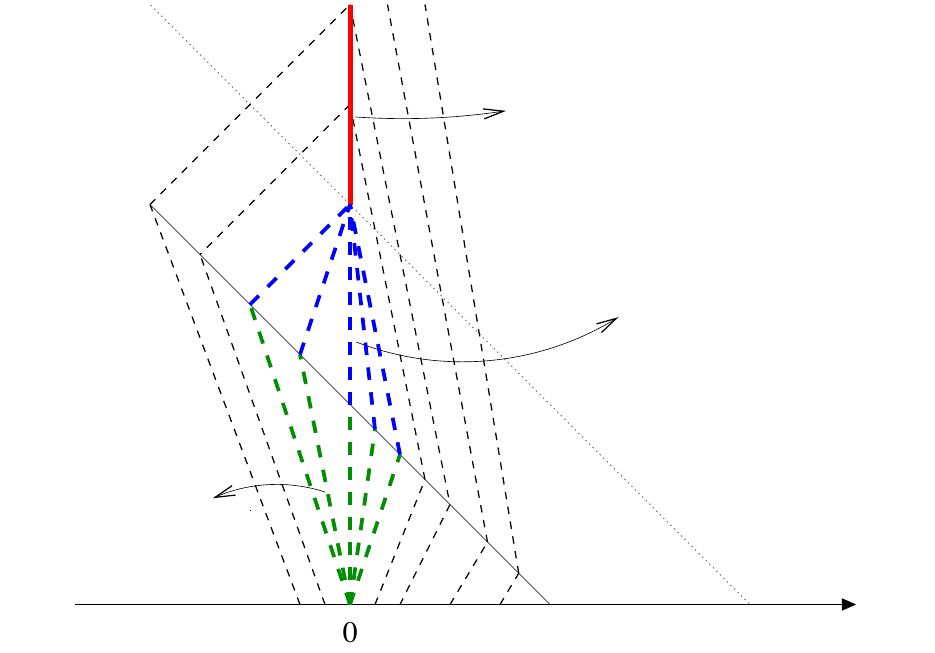}%
\end{picture}%
\setlength{\unitlength}{1579sp}%
\begingroup\makeatletter\ifx\SetFigFont\undefined%
\gdef\SetFigFont#1#2#3#4#5{%
  \reset@font\fontsize{#1}{#2pt}%
  \fontfamily{#3}\fontseries{#4}\fontshape{#5}%
  \selectfont}%
\fi\endgroup%
\begin{picture}(11203,7794)(-599,-6901)
\put(8251,-6811){\makebox(0,0)[lb]{\smash{{\SetFigFont{10}{12.0}{\familydefault}{\mddefault}{\updefault}{\color[rgb]{0,0,0}$2h$}%
}}}}
\put(5851,-6811){\makebox(0,0)[lb]{\smash{{\SetFigFont{10}{12.0}{\familydefault}{\mddefault}{\updefault}{\color[rgb]{0,0,0}$h$}%
}}}}
\put(10051,-6511){\makebox(0,0)[lb]{\smash{{\SetFigFont{10}{12.0}{\familydefault}{\mddefault}{\updefault}{\color[rgb]{0,0,0}$x$}%
}}}}
\put(5701,-511){\makebox(0,0)[lb]{\smash{{\SetFigFont{10}{12.0}{\familydefault}{\mddefault}{\updefault}{\color[rgb]{0,0,0}shock}%
}}}}
\put(6451,-2686){\makebox(0,0)[lb]{\smash{{\SetFigFont{10}{12.0}{\familydefault}{\mddefault}{\updefault}{\color[rgb]{0,0,0}compressive wave}%
}}}}
\put(-599,-5536){\makebox(0,0)[lb]{\smash{{\SetFigFont{10}{12.0}{\familydefault}{\mddefault}{\updefault}{\color[rgb]{0,0,0}rarefaction wave}%
}}}}
\end{picture}%

\caption{}
\label{f:construction_f}
\end{center}
\end{figure}

By the above construction, as shown in Figure \ref{f:construction_f}. The rarefaction wave turns out to be compressive in the area $\{h<x+t<2h\}$ which eventually generates a shock starting at the point $(0,2h)$ and propagating along the $x=0$. Notice that there is another shock starting from the point $(h,0)$. However, we can modify the initial data a little to get rid of this shock. In fact, recalling  the formula \eqref{e:value_u+} and letting
\[
 u_1=\frac{-1+\sqrt{1+2e^{-1/a}(f(u_r,-a)+u_r)}}{e^{-1/a}}.
\]
 We can replace $u_0$ in the initial data by
\[
 \tilde{u}_0=\begin{cases}
              u_l  & x<0,\\
              u_r  & 0<x<h,\\
              u_1  & x>h
             \end{cases}
\]
to make the speed of the characteristics starting from $\{h<x<2h\}$  the same as nearby. By the total variation estimates for the general system
\[
 {\rm Tot.Var.}\{u(\cdot,t)\}\lesssim {\rm Tot.Var.}\{u_0(\cdot)\},
\]
it is not restrict to assume that the total variation of $\tilde{u}_0$ is sufficiently small.

Now, we begin to find the initial data with which \eqref{e:example} had an admissible solution containing ``
Cantor set'' shocks.

Now, we  consider the initial data
\begin{equation}\label{e:initial_example}
 u_0(x)=\begin{cases}
         u_l & x<0,\\
         u_r & x>0,
        \end{cases}
\qquad v_0(x)=\begin{cases}
        0 & x<2h,\\
        -a &2h< x<3h,\\
       a & 3h<x<4h,\\
        0 & x>4h.
       \end{cases}
\end{equation}

Since the second equation in \eqref{e:example} is a linear transport equation, one has
\begin{equation}
v(x,t)=\begin{cases}
              0 & x+t<2h,\\
        -a &2h< x+t<3h,\\
       a & 3h<x+t<4h,\\
        0 & x+t>4h.
           \end{cases}
\end{equation}

In the area $\{x+t<2h\}$, the first equation in \eqref{e:example} becomes $u_t=0$, then one has
\begin{equation}
 u(x,t)=\begin{cases}
  u_l & x<\min\{2h-t,0\},\\
   u_r & x>\max\{2h-t,0\}
    \end{cases}
\end{equation}

Next, we compute the value of $u$ in $\{2h<x+t<3h\}$. For any $(x',t')\ne(0,2h)$ on the line $x+t=2h$, write
\[
 u^-=\lim_{\substack{(x,t)\to(x',t')\\ x+t<2h}} u(x,t),\qquad u^+=\lim_{\substack{(x,t)\to(x',t')\\ x+t>2h}} u(x,t).
\]
By Rankine-Hugoniot conditions, one has
\[
 f(u^+,-a)-f(u^-,0)=-(u^--u^-).
\]
which yields
\[
 u^-=u^++f(u^+,-a).
\]
Regarding $u^+$ as a function of $u^-$ and differentiating the above equation on both sides with respect to $u^-$, one gets
\[
 (f_u+1) (u^+)'=1.
\]
By \eqref{e:value_fu}, one has
\[
 (u^+)'=\frac{1}{1+f_u}=\frac{2g-1}{g}>0
\]
in a small neighborhood of origin. Thus $u^+$ is strictly increasing w.r.t. $u^-$. This concludes that at point $(0,2h)$ the left value of $u$ is still larger than the right value of $u$, i.e.
\[
 u^+_l>u^+_r,
\]
where
\[
 u^+_l=\lim_{\substack{(x,t)\to(0,2h)\\ 2h-t<x<0}} u(x,t),\qquad u^+_r=\lim_{\substack{(x,t)\to(0,2h)\\ x>\max\{2h-t,0\}}} u(x,t).
\]
As we discussed before, one gets a rarefaction wave by solving the Riemann problem in $\{2h<x+t<3h\}$ and a compressive waves in area $\{3h<x+t<4h\}$.

Next, we compute the value of $u$ in the zone $\{x+t>4h\}$. For any point $(x_0,t_0)\ne(0,4h)$ on the line $x+t=4h$, let
\[
 u^+(x_0,t_0)=\lim_{\substack{x+t>4h\\ (x,t)\to(x_0,t_0)}}u(x,t),\quad u^-(x_0,t_0)=\lim_{\substack{x+t<4h\\ (x,t)\to(x_0,t_0)}}u(x,t).
\]

By Rankine-Hugoniot conditions, one has
\[
 f(u^+,0)-f(u^-,a)=-(u^+-u^-).
\]
Then since $f(u^+,0)=0$ and $f(u^-,a)={e^{-1/a}(u^-)^2}/{2}$, one obtains
\[
 u^+=\frac{e^{-1/a}(u^-)^2}{2}+u^-.
\]

We claim that $u^+_L>u^+_R$, where
\[
 u^+_L=\lim_{\substack{(x,t)\to(4h,0)\\ 4h-t<x<0}} u(x,t),\qquad u^+_R=\lim_{\substack{(x,t)\to(4h,0)\\ x>\max\{4h-t,0\}}} u(x,t).
\]
are the left/right of $u$ across the line $x+t=4h$
In fact, since $u^-_L>u^-_R$, where
\[
u^-_L=\lim_{\substack{(x,t)\to(4h,0)\\ x+t<4h\\ x<\hat \lambda (t-4h) }} u(x,t),\qquad u^-_R=\lim_{\substack{(x,t)\to(4h,0)\\ x+t<4h\\ x>-\hat \lambda (t-4h)}} u(x,t).
\]
one has
\begin{align*}
 & u^+_L-u^+_R\\
=&\left(\frac{e^{-1/a}(u^-_L)^2}{2}+u^-_L\right)-\left(\frac{e^{-1/a}(u^-_R)^2}{2}+u^-_R\right)\\
=& (u^-_L-u^-_R)\left(1-\frac{e^{-1/a}(u^-_L)^2}{2}(u^-_L+u^-_R)\right)>0
\end{align*}
since $a$ and $(u^-_L,u^-_R)$ sufficiently small.

Therefore as the first equation in \eqref{e:example} is $u_t=0$ in $\{x+t>4h\}$, the jump prorogate along $x=0$ which turns out to be a shock. Similarly, by modifying the initial data a little, we can guarantee that there is not other shocks in the solution. (See Figure \ref{f:1_example}.)

\begin{figure}[htbp]
\begin{center}

\begin{picture}(0,0)%
\includegraphics{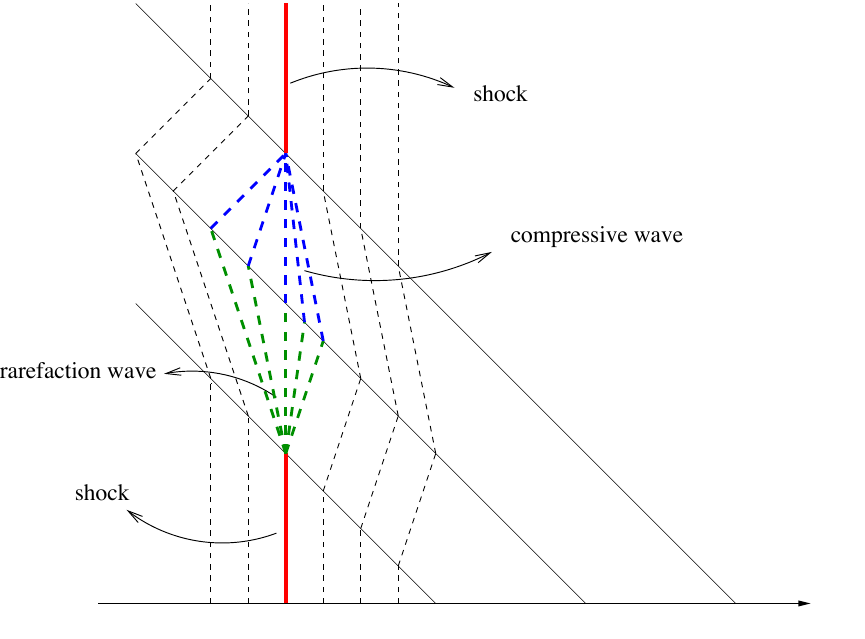}%
\end{picture}%
\setlength{\unitlength}{1184sp}%
\begingroup\makeatletter\ifx\SetFigFont\undefined%
\gdef\SetFigFont#1#2#3#4#5{%
  \reset@font\fontsize{#1}{#2pt}%
  \fontfamily{#3}\fontseries{#4}\fontshape{#5}%
  \selectfont}%
\fi\endgroup%
\begin{picture}(13799,10194)(-974,-9301)
\put(10726,-9211){\makebox(0,0)[lb]{\smash{{\SetFigFont{7}{8.4}{\familydefault}{\mddefault}{\updefault}{\color[rgb]{0,0,0}$6h$}%
}}}}
\put(8251,-9211){\makebox(0,0)[lb]{\smash{{\SetFigFont{7}{8.4}{\familydefault}{\mddefault}{\updefault}{\color[rgb]{0,0,0}$4h$}%
}}}}
\put(3451,-9211){\makebox(0,0)[lb]{\smash{{\SetFigFont{7}{8.4}{\familydefault}{\mddefault}{\updefault}{\color[rgb]{0,0,0}0}%
}}}}
\put(5851,-9211){\makebox(0,0)[lb]{\smash{{\SetFigFont{7}{8.4}{\familydefault}{\mddefault}{\updefault}{\color[rgb]{0,0,0}$2h$}%
}}}}
\put(12226,-8911){\makebox(0,0)[lb]{\smash{{\SetFigFont{7}{8.4}{\familydefault}{\mddefault}{\updefault}{\color[rgb]{0,0,0}$x$}%
}}}}
\end{picture}%

\caption{Afer modifying a little the initial data \eqref{e:initial_example}, the admissible solution contains two shocks along the $t$-axis.}
\label{f:1_example}
\end{center}
\end{figure}
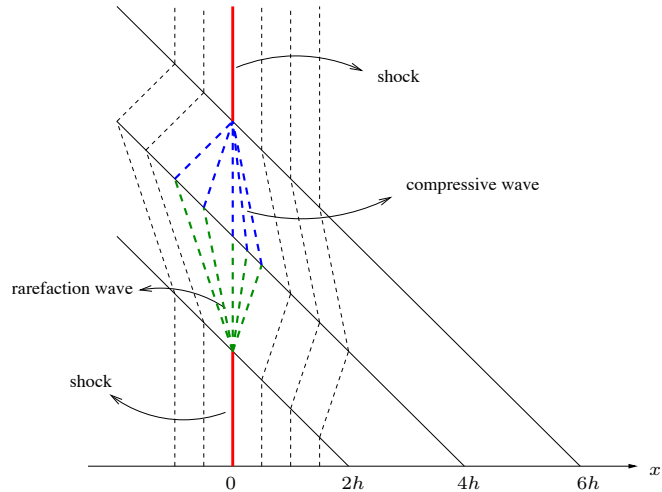

Next, we construct an admissible solution to such that it consists of graph of shock curves which is an 1-dimensional Cantor
set. Let
\[
 B^-_{m,n}=\left(\frac{3n+1}{3^m},\frac{3n+3/2}{3^m}\right)\cdot 6,\qquad B^+_{m,n}=\left(\frac{3n+3/2}{3^m},\frac{3n+2}{3^m}\right)\cdot 6, \quad n=0,1,\cdots,3^{m-1}-1.
\]
and
\[
 B_{m,n}=B^-_{m,n}\cup B^+_{m,n}\ \text{ and } B_m=\bigcup^{3^{m-1}-1}_{n=0}  B_{m,n}.
\]

We consider initial datum as
\begin{equation*}
 u_{0,m}(x)=\begin{cases}
         u_l & x<0,\\
         u_r & x>0,
        \end{cases}
\qquad v_{0,m} =\begin{cases}
         a_n & x\in B^-_{m,n},\\
        -a_n & x\in B^+_{m,n},\\
         0 & x\in\R\setminus B_{m,n},
        \end{cases}
\end{equation*}
where one can always choose $u_l,\ u_r$ and a suitable sequence $\{a_n\}$ such that total variation of $(u_0.v_0)$ is sufficiently small. By modifying $u_0$ properly similarly to get rid of extra shocks. One finds that the admissible solution $(u_m,v_m)$ of the system \eqref{e:example} is continuous inside $\hat{K}_m$ where
\[
 \hat{K}_m:=\bigcup^\infty_{m=1}\bigcup^{3^{m-1}-1}_{n=0} \{(x,t)\in \R^+\times\R:\ \frac{6(3n+1)}{3^m}-x<t<\frac{6(3n+2)}{3^m}-x  \}
\]
and its shocks belongs to 1-th family are located on
\[
\Big{\{}(0,t):t\in[0,\infty)\Big{\}}\setminus\bigcup^{3^{m-1}-1}_{k=0} \left\{(0,t):\  \frac{6(3k+1)}{3^m}<t<\frac{6(3k+2)}{3^m}\right\}.
\]

Obviously, $(u,v):=\lim_{m\rightarrow \infty} (u_m,v_m)$ is the admissible solution of \eqref{e:example} with initial data $(u_0,v_0):=\lim (u_{0,m},v_{0,m})$. As $m\rightarrow \infty$, $C_m:=[0,6]\setminus B_m$ converge to the Cantor set C (scaled by 6). Since the Cantor set is uncountable set and can not contain any interval of non-zero length, it is impossible to find countable many Lipschitz continuous curves which exactly cover discontinuities of $(u,v)$, that is, on these curves, either $(u,v)$ is discontinuous or it is an interaction point.This means that Theorem \ref{t:main theorem} fails for this situation.

\end{document}